\documentclass[11pt,psamsfonts,twoside]{article}
\usepackage{amsfonts}
\usepackage{bbding}
\usepackage{mathrsfs}
\usepackage{bbm}
\usepackage{amsmath}
\usepackage{amsmath,amsthm}
\usepackage{amssymb,amscd}
\usepackage{amsfonts,amsbsy}
\usepackage{fancyhdr,graphicx}
\usepackage[dvips]{psfrag}
\usepackage{indentfirst}
\numberwithin{equation}{section}
\newtheorem {proposition} {Proposition}[section]
\newtheorem {theorem}     [proposition]{Theorem}
\newtheorem {corollary}   [proposition]{Corollary}
\newtheorem {lemma}       [proposition]{Lemma}

\allowdisplaybreaks

\newcommand{\supp}{{\rm supp}}

\begin{document}
\setlength{\parindent}{4ex} \setlength{\parskip}{1ex}
\setlength{\oddsidemargin}{12mm} \setlength{\evensidemargin}{9mm}
%%---------------------------------------------------------------------------------------------------------
%                                          Title
%%---------------------------------------------------------------------------------------------------------
\title{{Well-posedness of a 3\hskip.02cm{D} Parabolic-hyperbolic Keller-Segel System in the Sobolev Space Framework}}

\author{Chao Deng, \quad\quad Tong Li}
\date{}

\maketitle
%% ------------------------------------------------------------------------------------------------------------------------------
%                                      Abstract
%%-------------------------------------------------------------------------------------------------------------------------------
\begin{abstract}
We study the global strong solutions to a 3-dimensional
parabolic-hyperbolic Keller-Segel model with initial data close to a
stable equilibrium with perturbations belonging to $L^2(\mathbb
R^3)\times H^1(\mathbb{R}^3)$. We obtain global well-posedness and
decay property. Furthermore, if the mean value of initial cell
density is smaller than a suitabale constant, then the chemical
concentration decays exponentially to zero as $t$ goes to infinity.
Proofs of the main results are based on an application of Fourier
analysis method to uniform estimates for a linearized
parabolic-hyperbolic system and also based on the smoothing effect
of the cell density as well as the damping effect of the chemical
concentration.

\end{abstract}

%\bigskip

%\end{abstract}
%%-------------------------------------------------------------------------------------------------------------------------------
\maketitle
%%-------------------------------------------------------------------------------------------------------------------------------
%                                     Introduction
%%-------------------------------------------------------------------------------------------------------------------------------
\section{Introduction}
 In this paper, we study the following normalized 3-dimensional chemotaxis model
\begin{align}\label{eq:1.1}
  \left\{\begin{aligned}
  &\partial_tu=\Delta u +\nabla\cdot(u\nabla\ln v),\ \ \ \ \ \ \ \ \  \\
  &\partial_tv=uv-\mu v\end{aligned}\right.
 \end{align}
 for $t>0$ and $x\in \mathbb{R}^3$, where $u(x,t)$, $v(x,t)$ denote the cell density and the chemical concentration, respectively.
 System \eqref{eq:1.1} was proposed by Othmer and Stevens \cite{O-S}
 to describe the chemotactic movement of particles where  the chemicals are non-diffusible and can modify the local environment for succeeding
 passages. For example, myxobacteria produce slime over which their cohorts can move more readily and ants can follow trails left by
 predecessors \cite{E}. One direct application of \eqref{eq:1.1} is to model haptotaxis where cells move towards an increasing
 concentration of immobilized signals such as surface or matrix-bound adhesive molecules.

 With no loss of generality, by setting $w=\mu{t}+\ln{v}$ %i%and
 %applying this transformation to
 in \eqref{eq:1.1}, we get
 \begin{align}\label{eq:1.2}\left\{
  \begin{aligned}
  &\partial_tu=\Delta u+\nabla\cdot(u\nabla{w}),\\
  & \partial_t {w}=u,\\
  & (u,w)|_{t=0}=(u_0,w_0)
  \end{aligned}\right.\hskip1.18cm
  \end{align}
 for $t>0$ and $x\in \mathbb{R}^3$. %System
% \eqref{eq:1.2} is supplemented with initial data
%  \begin{equation}\label{Initialdatauw}
%  (u, w)(0,x)=(u_0, w_0)(x)
%   \end{equation}
%    for $x\in \mathbb{R}^3$. %,  where $(p_0,q_0)\in H^k(\mathbb R^3)\times{H}^1(\mathbb R^3)$ ($k=0,2$).
 System \eqref{eq:1.2} was studied in \cite{W-H}
 in one-dimensional case and was extended to
 multidimensional cases in \cite{L-Z,LiPanZhao:2012}. It was studied in \cite{O-S}
 and a comprehensive qualitative and numerical analysis was provided there. We refer readers to Refs.
 \cite{CheminLerner:1992314, C-P-Z1,C-P-Z,E,Horstman,KS,KS1,KS2,L-S,Likun_prep,L-Z,L-W,L-W1,L-W2,N-I,
 Patlak,S-W-W,W-H,Y-C-L,Y-C-L-S,Z-Z} for more discussions in this direction.
   Recently, in \cite{L-Z}, the local and global existence
 of the classical solution to \eqref{eq:1.2} was studied when initial
 data $(u_0-\bar{u}, \nabla{w}_0)\in H^{\frac{5}{2}+}(\mathbb R^3)\times{H^{\frac{5}{2}+}}(\mathbb R^3)$ with $\bar{u}$ being the mean value of $u_0$ ($s=\frac{5}{2}+$ stands for
 $s>\!\frac{5}{2}$ and similar conventions are applied throughout this
 paper).
 Later on, Hao \cite{Hao:2012} studied global existence and uniqueness of global mild solution for initial data
  close to some constant state in critical Besov space with minimal regularity
 where the proof is in the Chemin-Lerner space framework which was introduced by Chemin and
 Lerner \cite{CheminLerner:1992314} and aferwards developed in a series of works (see e.g. \cite{Danchin:2000579}).

 Noticing that the Cauchy problem of system \eqref{eq:1.2} is invariant under the following scaling %invariant
 transformations
 $$\Big(u(t,x),\; (\nabla{w})(t,x)\Big)\rightarrow \Big(\lambda^2 u(\lambda^2t,\lambda x),\;\lambda({\nabla}w)(\lambda^2t,\lambda{x})\Big)$$
 and
  $$\Big(u_0(x),\; ({\nabla}w_0)(x)\Big)\rightarrow \Big(\lambda^2 u_0(\lambda x),\;\lambda({\nabla}w_0)(\lambda{x})\Big).$$
  The idea of using a functional setting invariant by the scaling is now classical and was originated from many
  works (see e.g. \cite{Cannone:1995}).
  It is clear that the critical Sobolev space for $(u_0,\nabla{w}_0)$ is
  $\dot{H}^{\hskip.01cm-\hskip.01cm\frac{1}{2}}(\mathbb R^3)\times \dot{H}^{\frac{1}{2}}(\mathbb R^3)$ and,
  correspondingly, $\dot{H}^{s-2}(\mathbb R^3)\times\dot{H}^{s-1}(\mathbb R^3)$ ($s>\frac{3}{2}$)
  is the subcritical Sobolev space.

  As for the critical case, it seems to be difficult to prove global existence of mild solution to system \eqref{eq:1.2}
  with $(u_0,\nabla{w}_0)\in\dot{H}^{-\frac{1}{2}}(\mathbb R^3)\times
 \dot{H}^{\frac{1}{2}}(\mathbb R^3)$ due to the invalidity of $\dot{H}^{\frac{3}{2}}(\mathbb R^3)\!\hookrightarrow\!L^\infty(\mathbb R^3)$.
 Thus a suitably smaller initial data space$--$the hybrid Besov space
 $\dot{B}^{-\frac{1}{2}}_{2,1}(\mathbb{R}^3)\times(\dot{B}^{\frac{1}{2}}_{2,1}(\mathbb{R}^3)\cap\dot{B}^{\frac{3}{2}}_{2,1}(\mathbb{R}^3))$
  and $\dot{B}^{\frac{3}{2}}_{2,1}(\mathbb{R}^3)\!\hookrightarrow\! L^\infty(\mathbb{R}^3)$ were used in \cite{Hao:2012}.

 As for the subcritical case, we observe that $L^2(\mathbb R^3)$ function $(1+|x|^2)^{\hskip-.01cm-\hskip-.01cm1}$ neither
 belongs to $\dot{H}^{\hskip-.01cm-\hskip-.01cm\frac{1}{2}}(\mathbb R^3)$ nor to
 $\dot{B}^{\hskip-.02cm-\hskip-.01cm\frac{1}{2}}_{2,1}(\mathbb R^3)$. Hence the case of $(u_0,\nabla{w}_0)\in\! L^2(\mathbb R^3)\times H^1(\mathbb R^3)$
 can not be treated directly by applying results of the critical case, cf. \cite{Hao:2012}.
 %Therefore, our results can be viewed as a compensate.
 We believe that the Chemin-Lerner space framework can be modified
 slightly to handle the subcritical cases. However, we do not proceed to this way but consider well-posedness of mild
 solution in the $L^2(\mathbb R^3)\times H^1(\mathbb R^3)$ framework and Fourier multiplier theory
 provides us with another option. Recalling the well known weak solution theory for
 heat equation, we observe that searching a solution $u$ in $C([0,\infty);L^2(\mathbb R^3))\cap L^2(0,\infty;\dot{H}^1(\mathbb R^3))$
 is also important to understand \eqref{eq:1.2}. Meanwhile, assuming  $\nabla{w}_0\in H^1(\mathbb R^3)$ is convenient to study the decay property of $w$.

 Following similar energy arguments as in \cite{L-Z}, one can decrease the indices of
 the solution space $H^{s}(\mathbb R^3)$ from $s=\frac{5}{2}+$ to $s=2$ or even
 to $s=\frac{3}{2}+$, where $s=\frac{3}{2}$ seems to be unreachable for energy arguments.
 %to $s>\frac{1}{2}$.
 Indeed, by multiplying \eqref{eq:1.2} by some proper terms of $u$ and $w$, then integrating by parts, we get
\begin{align*}%\label{eq:energy}
   &\frac{d}{dt}\Big(\|u\|_{L^2}^2\!+\!\|\nabla w\|_{L^2}^2\!+\!\|\Lambda^{{\frac{3}{2}+}}u\|_{L^2}^2\!+\!\|\Lambda^{{\frac{3}{2}+}}\nabla{w}\|_{L^2}^2\Big)
      \!\le\!-2\|{\nabla}u\|_{L^2}^2+\|\nabla{(u^2)}\cdot\nabla{w}\|_{L^1}\nonumber\\
         &\;\;+\!2\|\nabla{w}\cdot\nabla{u}\|_{L^1}
        -\!2\|\Lambda^{{\frac{5}{2}+}}u\|_{L^2}^2+\|\nabla\Lambda^{{\frac{3}{2}+}}\!p\cdot\!\Lambda^{\frac{3}{2}+}(pq)\|_{L^1}\!+\!
        \|\nabla\Lambda^{\frac{3}{2}+}\!{u}\cdot\!\nabla\Lambda^{\frac{3}{2}+}w\|_{L^1}.
         \end{align*}
 %\begin{align*}
   %\frac{d}{dt}(\|u\|_{L^2}^2)\le-2\|{\nabla}u\|_{L^2}^2+\min\{\|\nabla{(u^2)}\cdot\nabla{w}\|_{L^1},\;\|u^2\Delta{w}\|_{L^1}\}.
 %\end{align*}
 Applying H\"older's inequality and $H^{\frac{3}{2}+}(\mathbb{R}^3)\!\hookrightarrow\!L^{\!\infty}(\mathbb{R}^3)$  to the above inequality
  and following similar arguments of the proof of Theorem 1.1 in \cite{L-Z}, we get
   \begin{align*}
     \frac{d}{dt}(\|u\|_{H^{\frac{3}{2}+}}^2+\|\nabla w\|_{H^{\frac{3}{2}+}}^2+1)\lesssim -\|\nabla{u}\|_{H^{\frac{3}{2}+}}^2
       +(\|u\|_{H^{\frac{3}{2}+}}^2+\|\nabla{w}\|_{H^{\frac{3}{2}+}}^2+1)^2. % \|\nabla{(u^2)}\cdot\nabla{w}\|_{L^1}\lesssim\|\nabla{u}\|_{L^2}\|u\|_{L^2}\|\nabla{w}\|_{H^{s}}
         \end{align*}
          Then a simple Gronwall argument yields local
       well-posedness of \eqref{eq:1.2} for $(u_0,\nabla w_0)\in H^{\frac{3}{2}+}(\mathbb{R}^3)\times
        H^{\frac{3}{2}+}(\mathbb{R}^3)$. The methods used in the present article serve as a supplement to the energy method.
 %Applying H\"older's inequality, Sobolev embedding and interpolation \cite{BerghLofstrom:1976}, we get
 %\begin{align*}
%\left\{\begin{aligned}
 %&\|\nabla{(u^2)}\cdot\nabla{w}\|_{L^1}\lesssim\|\nabla{u}\|_{L^2}^{\frac{5}{2}-s_1}\|u\|_{L^2}^{{s_1-\frac{1}{2}}}\|\nabla{w}\|_{H^{s_1}},\\
 %&\|\nabla{(u^2)}\cdot\nabla{w}\|_{L^1}\lesssim\|\nabla{u}\|_{L^2}\|u\|_{L^2}\|\nabla{w}\|_{H^{s_2}},\\
 %&\|u^2\Delta{w}\|_{L^1}\lesssim\|u\|_{L^2}^2\|\Delta{w}\|_{L^\infty}\lesssim\|u\|_{L^2}^2\|\nabla{w}\|_{H^{s_3}}\;
 %\end{aligned} \quad\quad \quad\quad \quad\;\right.
 % \end{align*}
 %with $\frac{1}{2}<s_1<\frac{3}{2}<s_2$ and $s_3>\frac{5}{2}$. %Estimates
 %for %$w$ in $H^s(\mathbb R^3)$ norms and
 %for $(u,w)$ in $H^s(\mathbb R^3)$ norms follow from standard energy method.

 Meanwhile, in system \eqref{eq:1.2}, one needs to consider two major terms
 $\Delta{u}$ and $u\Delta{w}$. It suffices to assume that all the second derivatives of $u$ and
 $w$ exist almost everywhere, although maybe certain higher
 derivatives will not exist.
 \noindent Consequently, we also expect to establish
 well-posedness of %almost everywhere
 such solution to system \eqref{eq:1.2} with initial data $(u_0,\nabla{w}_0)\in H^2(\mathbb R^3)\times
 H^1(\mathbb R^3)$. Precisely, we will show that the Cauchy problem
 of system \eqref{eq:1.1} has a unique solution $(u-\bar{u},\nabla((\mu-\bar{u})t+\ln v))$ in
 $C([0,\infty);H^2(\mathbb R^3))\times C([0,\infty);H^1(\mathbb R^3))$ provided that the initial data $(u_0,\nabla\ln v_0)$ is close to some constant equilibrium
 state $(\bar{u}, 0)$ and the difference $(u_0-\bar{u},\nabla\ln{v}_0)$ belongs to $\in{H}^2(\mathbb R^3)\times{H}^1(\mathbb R^3)$ ($\bar{u}$ is defined
 in \eqref{eq:1.3}).

  In the 4-dimensional case, scaling invariant
  discussion suggests that the initial data space $L^2(\mathbb R^4)\times H^1(\mathbb R^4)$ is critical.
  It is an interesting {\it question} whether
  the 4D model \eqref{eq:1.2} has a solution even locally in time with $(u_0,\nabla w_0)\in L^2(\mathbb R^4)
  \times{H}^1(\mathbb R^4)$.

 By modifying the definition of the mean value of $u$ in bounded domains, we define
 \begin{align}\label{eq:1.3}
    \bar{u} =
    \lim_{R\rightarrow\infty}\frac{1}{|B_R|}\int_{B_R}u_0(x)dx,\hskip1cm
 \end{align}
 where $B_R\subset \mathbb{R}^3$ is a ball centered at the origin with radius
 $R$ and $u_0$ is the initial cell density. Applying $p=u-\bar{u}$,
 $h=(\mu-\bar{u})t+\ln{v}$ and $\bar{u}=1$  to \eqref{eq:1.1}, we get
 \begin{align*}
 \left\{\begin{aligned}
 &\partial_tp=\Delta p + \Delta h + \nabla\cdot (p\nabla h),\\
 &\partial_th=p.\end{aligned}\right.\hskip.81cm
 \end{align*}

 It is easy to check that for any positive constant $c$, if $(p,h)$ is a solution to the above system,
 then $(p,h+\ln c)$ is also a solution. Or equivalently, if
 $(u,v)$ is a solution to system \eqref{eq:1.1}, then $(u,cv)$ is also a
 solution to system \eqref{eq:1.1}.
% {\it This might be interpreted as different cells have different
% sensitivity to different chemicals and $c$ is an index of  sensibility.}
 It is natural to think
 $\nabla{h}$ as a new unknown function whence $\nabla{h}$ is
 uniquely determined.
 Setting $\Lambda=\sqrt{-\Delta}$,
 $q=-\Lambda h$ and $G=\Lambda^{-1}\nabla\cdot(p\nabla\Lambda^{-1}q)$, we
 obtain the following model
 \begin{align}\label{eq:1.4}
  \left\{
  \begin{aligned}
  &\partial_tp=\Delta p + \Lambda q - \Lambda G,\ \ \ \\
  &\partial_tq=-\Lambda p\ \
  \end{aligned}
  \right.\hskip1.59cm
 \end{align}
 for $t>0$ and $x\in \mathbb{R}^3$.

 In this paper, we study system \eqref{eq:1.4} with initial data
 $(p_0(x), q_0(x))\in H^k(\mathbb R^3)\times{H}^1(\mathbb R^3)$ ($k=0,2$).
 More precisely, we prove the global well-posedness of system \eqref{eq:1.4} with small
 initial data satisfying $(p_0,q_0)\in H^k(\mathbb R^3)\times{H}^1(\mathbb R^3)$
 ($k=0,2$, see Theorems \ref{thm:1.1} and \ref{thm:1.2}). The main tools are Fourier transformation theory and the smoothing properties of
 parabolic-hyperbolic coupled systems (see inequalities \eqref{eq:m1}--\eqref{eq:m3} below for details). Particularly,
 from \eqref{eq3,6} and \eqref{eq:m1} as well as definition of $m_1(t,\xi)$ for $|\xi|>2$ in ($M$), we observe
 that if $|\xi|>4$, then
 \begin{align}\label{smoothing-1}
 m_1(t,\xi)=\frac{e^{-\frac{t(1+\Xi)|\xi|^2}{2}}}{\frac{2\Xi}{\Xi+1}}-\frac{e^{-\frac{2t}{1+\Xi}}}{\frac{\Xi(\Xi+1)|\xi|^2}{2}}
 \quad\text{ with }\;
 \Xi=\sqrt{1-\frac{4}{|\xi|^2}}\in(\frac{\sqrt{3}}{2},1).
 \end{align}
 Considering the smoothing effects, we need to study $\partial_t^k\partial^\alpha
 m_1(t,D)$ with symbol
 \begin{align}\label{smoothing-2}
 \partial_t^{k}\xi^\alpha
 m_1(t,\xi)=-\frac{(1\!+\!\Xi)^{k+1}}{{(-2)^{k+1}\Xi}}|\xi|^{2k}\xi^\alpha{e^{-\frac{t(1+\Xi)|\xi|^2}{2}}}+
   \frac{(-2)^{k+1}}{{\Xi(\Xi+1)^{k+1}}}{\xi^\alpha}{|\xi|^{-2}}{e^{-\frac{2t}{1+\Xi}}},
   \end{align}
  $\xi^\alpha=\xi_1^{\alpha_1}\xi_2^{\alpha_2}\xi_3^{\alpha_3}$,
 $\alpha=(\alpha_1,\alpha_2,\alpha_3)\in\mathbb{N}^3$,
 $k\in\mathbb{N}$ and $|\alpha|=\alpha_1+\alpha_2+\alpha_3\le2$. Indeed, for
 any $p_0\in L^2\!,$ $t>0$ and $|\xi|>4$, from \eqref{smoothing-1} and
 \eqref{smoothing-2} we have the following smoothing property
 \begin{align}\label{smoothing-3}
 \|\partial_t^k\partial^\alpha m_1(t,D)p_0\|_{L^2}\le
 \Big(C_1(k)\,t^{-\frac{|\alpha|}{2}-k}+C_2(k)\,
 e^{-t}\Big)\|p_0\|_{L^2}.
 \end{align}
 However, following similar arguments as in
 \eqref{smoothing-1}--\eqref{smoothing-3}, if $t>0$, $|\xi|>4$ and
 $q_0\in H^1$, then from \eqref{eq3,6}, ($M$) and \eqref{eq:m1} we can only get
 \begin{align}\label{smooting-4}
   \|\partial_t^k m_2(t,D) q_0\|_{H^1}\le C_3(k)\|q_0\|_{H^1},
 \end{align}
 where no smoothing effect exists for spatial variable. Considering the
 low frequency piece, smoothing properties of $m_1(t,D)$ and
 $m_2(t,D)(-\Delta)^{-\frac{1}{2}}$ are similar to $e^{t\Delta}$,
 hence it is omitted. %left to the readers. i
 In some cases, this special coupled system with such smoothing
 effects is also called weak dissipative structure, see for instance \cite{Lei:201013}. For various aspects of the smoothing
 properties, we refer the readers to see, for instance \cite{ChenZhang:20061793,Danchin:2000579}
 and the references therein.
 The proof here is based on a combination of the Fourier transform and estimates of the eigenvalues of the corresponding characteristic matrix
 (see \eqref{eq3,1}--\eqref{eq:m3} below for details). The different decay properties of the eigenvalues of the characteristic matrix
 enable us to
 take advantages of the smoothing property of the {\it high} frequency piece\,\footnote{Definitions of the low, medium and high frequency pieces of a function are given by \eqref{eq:1.7}.}
  of $p$, i.e., $p\in
 L^1(0,\infty;\dot{H}^{{7}/{4}}_\psi)$ instead of that of $q$ since the high frequency piece of $q$ does not have spatial smoothing effect
 (see  \eqref{smooting-4} above).
 The introduced $L^1(0,\infty;\dot{H}^{{7}/{4}}_\psi)$ space is the {\it new} point of this article.
 The main difficulty is to estimate $\|p\nabla{q}\|_{L^1(0,\infty;L^2)}$, which forces us to use frequency
 decomposition or partition of unit
 and smoothing effect of the high frequency piece of $p$ (see Lemma
 \ref{lem:3.2} below). Once $\|p\nabla{q}\|_{L^1(0,\infty;L^2)}$ being estimated, the desired result follows from a standard fixed point argument.
 As for the decay property of $v$ in system \eqref{eq:1.1}, we apply
 the limiting case of the Sobolev inequality in $BMO$ (cf. for instance, \cite{KozonoTaniuchi:2000191}) to
 $v=c\,e^{(\bar{u}-\mu)t}e^{-\Lambda^{-1}q}$ hence obtain lower and upper bounds for its $L^\infty$ norm which are stated in
 \eqref{eq:1.14}--\eqref{eq:1.15}.

 Before stating the main results, we define the partition of unit. Let us briefly explain how
 it may be built in $\mathbb{R}^3$. Let $\mathcal{S}(\mathbb{R}^3)$ be the
 Schwarz class and $(\eta,\varphi,\psi)$ be three smooth radially symmetric functions valued in
 $[0,1]$ such that
  \begin{align}\label{eq:1.5}
  & \supp\;\psi \subset \{\xi\in\mathbb{R}^3;\; |\xi|> 2^4\},         & \supp\;\varphi \subset \{\xi\in\mathbb{R}^3;\; 1<|\xi|< 2^5\}, \\
  & \supp\;\eta \subset\{\xi\in\mathbb{R}^3;\;|\xi|< 2\},             & \eta(\xi)+\varphi(\xi)+\psi(\xi)=1,\ \forall\ \xi\!\in\!\mathbb{R}^3.
  \label{eq:1.6}
  \end{align}
 For $f\in\mathcal{S}'(\mathbb{R}^3)$, we define the low, medium and high frequency operators as follows:\footnote{
 $f^l=\eta(D)f=\mathcal{F}^{-1}(\eta(\xi)\widehat{f}(\xi))$ and similar conventions are applied throughout this paper.}
\begin{align}\label{eq:1.7}
  f^l=\eta(D){f},\hskip.4cm
  f^m=\varphi(D){f},\hskip.4cm
  f^h=\psi(D){f},\hskip.4cm
  \eta(D)\psi(D){f}\equiv0
\end{align}
 with $\eta(\xi),\varphi(\xi)$ and $\psi(\xi)$  being symbols of $\eta(D),\varphi(D)$ and $\psi(D)$, respectively.

 Throughout this paper,  $\mathcal{F}f$ and $\widehat{f}$ stand for Fourier
 transform of $f$ with respect to space variable and
 $\mathcal{F}^{-1}$ stands for the corresponding inverse Fourier transform. For any $s\ge0$ and any function $f$, we shall define the
 fractional Riesz potential $\Lambda^s$ and Bessel potential $\langle \Lambda \rangle^s:=(1-\Delta)^{\frac{s}{2}}$
 via
 \begin{align}\label{eq:1.8}
   \widehat{\Lambda^s f}(\xi)=|\xi|^s\widehat{f}(\xi) \ \text{ and }\   \widehat{\langle\Lambda\rangle^s f}(\xi)=\langle\xi\rangle^s\widehat{f}(\xi)=
   ({1+|\xi|^2})^\frac{s}{2}\widehat{f}(\xi),
 \end{align}
 respectively. $\|\cdot\|_{L^2}$, $\|\cdot\|_{L^\infty}$, $\|\cdot\|_{H^s}$ and $\|\cdot\|_{\dot{H}^s}$ denote the
 norms of the usual Lebesgue measurable function spaces $L^2$, $L^\infty$, the
 usual Bessel potential space
 \begin{align} \label{eq:1.9}
 H^s:=\{f\in\mathcal{S}'(\mathbb{R}^3);\ \|\langle \Lambda\rangle^sf\|_{L^2}<\infty \}
 \end{align} and Riesz potential space
 \begin{align}\label{eq:1.10}
 \dot{H^s}:=\{f\in\mathcal{S}'(\mathbb{R}^3);\ \|\Lambda^sf\|_{L^2}<\infty \},
 \end{align}
 respectively. Moreover, from \eqref{eq:1.9} and \eqref{eq:1.10}, we observe that for any $s>0$,
 there holds  $H^s=\dot{H}^s\cap L^2$.
 For simplicity, for any $s\in\mathbb{R}$, we define
 \begin{align}\label{eq:1.11}
 \dot{H}^{s}_{\psi}=\{f\in \mathcal{S}'(\mathbb{R}^3);\
 \|f\|_{\dot{H}^s_{\psi}}=\|\Lambda^s\psi(D)f\|_{L^2}=\|\Lambda^sf^h\|_{L^2}<\infty\},
 \end{align}
 where $\dot{H}^s_\psi$ itself is not a Banach space since from \eqref{eq:1.7} one can prove that for any $g\in\mathcal{S}(\mathbb{R}^3)$
 satisfying $\supp\, \widehat{g}\subset\{\xi\in\mathbb{R}^3;\;|\xi|<2^4\}$ and $f\in\dot{H}^s_{\psi}$, there holds
 $\|f\|_{\dot{H}^s_\psi}=\|f+g\|_{\dot{H}^s_{\psi}}$. Hence we need to introduce another Banach space $Z$ to get an intersection space $Z\cap\dot{H}^s_\psi$ which forms a Banach space.

 The function space $C([0,\infty);X)$ is equipped with norm $\|f\|_{L^\infty_tX}$, where $X$ stands for some Banach
 space. For any two quantities $A$ and $B$, we shall use the notation
 $A\lesssim B$ when $A \le CB$ for some %inessential
 positive constant $C$. The dependence of $C$ on
 various parameters is usually clear from the context. $A\sim B$
 if and only if $A\lesssim B$ and $B\lesssim A$. For any $1\le{\rho}, r\le\infty$, we denote
 $L^\rho(0,\infty)$ and $L^\rho(0,\infty;L^r)$ by
  $L^\rho_t$ and  $L^\rho_tL^r$, respectively.

 We state the main results as follows.
 \begin{theorem}\label{thm:1.1}
  For any initial data $(p_0,{q}_0)\in L^2(\mathbb R^3)\times {H}^1(\mathbb R^3)$,
   there exist positive constants $C$ and $\varepsilon_0$ such that if $\|(p_0,  {q}_0)\|_{L^2\times H^1}\le\varepsilon_0$, then
   system \eqref{eq:1.4} has a unique global solution $\,\,(p,q)\in{C}([0,\infty);L^2)\times C([0,\infty);H^1)\,\,$ satisfying
  \begin{align*}
  \|(p,\,q)\|_{L^\infty_tL^2\times{L^\infty_tH^1}}+\|(\nabla{p},\nabla{q})\|_{L^2_tL^2\times{L^2_tL^2}}+\|p\|_{L^1_t\dot{H}^{\frac{7}{4}}_\psi}\le
   {C}\,\varepsilon_0.
   \end{align*}
  \end{theorem}
%  Applying smoothing effect of heat kernel, we can increase the regularity of space variable
%  by decreasing the time integrability which is more or less standard now.
%  As mentioned before, the regular index of $q_0$ can be lowered
%  down to $s>\frac{1}{2}$ and the corresponding regular index of $p_0$
%  could be lowered down to $s>-\frac{1}{2}$ with only minor modifications for their proofs.
%  As far as the authors know, well-posedness of system \eqref{eq:1.4} in critical Sobolev space $\dot{H}^{-\frac{1}{2}}\times
%  \dot{H}^{\frac{3}{2}}$ is very difficult, no matter using the Chemin-Lerner space framework or using the skills used
%  in the present article. However, this paper emphasize this difficult again and thus it is left as an open question where similar
%  situations exist for many other PDEs including the compressible
%  Navier-Stokes equations and viscoelastic fluids models. %The already known well-posedness results in critical case are done in the Besov space framework.
  \begin{theorem}\label{thm:1.2}
  For any initial data $(p_0,{q}_0)\in H^2(\mathbb R^3)\times {H}^1(\mathbb R^3)\!,$
  there exist positive constants $C$ and $\varepsilon_0$ such that
  if $\|(p_0,{q}_0)\|_{H^2\times H^1}\le\varepsilon_0$, then system
   \eqref{eq:1.4} has a unique global solution $\,\,(p,q)\in{C}([0,\infty);H^2)\times{C}([0,\infty);H^1)\,\,$
   satisfying
  \begin{align*}
  \|(p,q)\|_{L^{\!\infty}_tH^2\times{L^{\!\infty}_tH}^1}+\sup_{t>0}\,(1\!+t)^{\frac{1}{2}}\|(\nabla{p},\nabla{q})\|_{L^2\times L^2}+\sup_{t>0}\,(1\!+t)^{\frac{7}{8}}\|\Lambda^{\frac{7}{4}}p\|_{L^2}
  \le{ C}\varepsilon_0.
  \end{align*}
   \end{theorem}
  Theorem \ref{thm:1.2} proves well-posedness of \eqref{eq:1.4} with data $(p_0,q_0)\!\in\!{H}^2(\mathbb R^3)\!\times\!{H}^{\hskip-.015cm1}(\mathbb
 R^3).$ Notice that $L^\infty$ is also natural setting for cell density $u$ and chemical concentration $v$.
 Based on the transformation of $(u,v)$ and $(p,q)$, in order to study the $L^\infty$ norm decay of $(u,v)$, we
 add ${\displaystyle{\sup_{t>0}}}\,(1\!+t)^{\frac{1}{2}}\|(\nabla{p},\nabla{q})\|_{L^2\times{L}^2}$ and
 ${\displaystyle{\sup_{t>0}}}\,(1\!+t)^{\frac{7}{8}}\|\Lambda^{\frac{7}{4}}p\|_{L^2}$.

 %Recalling the transformations stated in \eqref{eq:1.1}--\eqref{eq:1.4} and w
 Recall that if
 $(u,v)$ solves the system \eqref{eq:1.1}, then for any positive constant $c$,
 $(u,cv)$ also solves system \eqref{eq:1.1}. Hence from the unique solution $(p,q)$ of \eqref{eq:1.4}, we
 have a sequence of solutions $(u,cv)$ such that
  $v=c\,e^{(\bar{u}-\mu)t}e^{-\Lambda^{-1}q}$. Keeping this in mind  and from embedding theorems
 $\dot{H}^{\frac{5}{4}}\hookrightarrow{L}^{12}$, $\dot{H}^\frac{1}{2}\hookrightarrow
 L^3\hookrightarrow{BMO}^{-1}$ as well as Lemma \ref{lem2.6} below, we get
  $v=c\,e^{(\bar{u}-\mu)t}e^{-\Lambda^{-1}q}$ and\,\footnote{
  We refer the readers to \cite{Stein:1993} to see definition of $BMO$ space and \cite{KochTataru:200122} to see definition of $BMO^{-1}$
  as well as embedding theorem $L^n \hookrightarrow BMO^{-1}$.}
 \begin{align}
 \|\Lambda^{-1}q\|_{L^\infty} &\le {C}(1+\|\Lambda^{-1}q\|_{BMO}(1+\max\{0,\ln{\|\Lambda^{-1}q\|_{W^{\frac{3}{4},12}}}\})\,)\nonumber\\
                              &\le {C}(1+\|q\|_{BMO^{-1}}(1+\max\{0,\ln(\|\Lambda^{-1}q\|_{L^{12}}+\|\Lambda^{-\frac{1}{4}}q\|_{L^{12}})\})\,)\nonumber\\
                              &\le {C}(1+\|q\|_{\dot{H}^{\frac{1}{2}}}(1+\max\{0,\ln(\|q\|_{\dot{H}^{\frac{1}{4}}}+\|q\|_{\dot{H}^{1}})\})\,)\nonumber\\
                              &\le
                              {C}(1+\|\nabla{q}\|_{L^2}^{\frac{1}{2}}\|q\|_{L^2}^\frac{1}{2}(1+\max\{0,\ln\|q\|_{{H}^{1}}\})\,)\label{eq:1.12}.
 \end{align}
 In Theorem \ref{thm:1.2}, we chose $\varepsilon_0$ such that $C\varepsilon_0\le 1$. Then from \eqref{eq:1.12},
 we obtain that
 \begin{align}\label{eq:1.13}
 \|\Lambda^{-1}q\|_{L^\infty}\le
 {C}(1+\|\nabla{q}\|_{L^2}^{\frac{1}{2}}\|q\|_{L^2}^\frac{1}{2})\le{C}(1+\|\nabla{q}\|_{L^2}^{\frac{1}{2}}).
 \end{align}
  Making use of definition of Banach valued series $e^{f}$, we
 observe that $e^{f}$ is well defined if $f\in L^\infty$. Applying
 \eqref{eq:1.13} to $v=c\,e^{(\bar{u}-\mu)t}e^{-\Lambda^{-1}q}$, we get
  \begin{align}
    \label{eq:1.14}
    \frac{1}{c}\|v\|_{L^\infty}  & =  e^{(\bar{u}-\mu)t}\|e^{-\Lambda^{-1}q}\|_{L^\infty} \le {e}^{(\bar{u}-\mu)t}\,e^{\,\|\Lambda^{-1}q\|_{L^\infty}}
     \le
     {e}^{(\bar{u}-\mu)t}\,e^{\,C(1+\|\nabla{q}\|_{L^2}^{\frac{1}{2}})},
  \\
    \label{eq:1.15}
    \frac{1}{c}\|v\|_{L^\infty}  & =  e^{(\bar{u}-\mu)t}\|e^{-\Lambda^{-1}q}\|_{L^\infty}
                        \ge {e}^{(\bar{u}-\mu)t}e^{-\|\Lambda^{-1}q\|_{L^\infty}}\!\ge\!  {e}^{(\bar{u}-\mu)t}{e}^{-C(1+\|\nabla{q}\|_{L^2}^{\frac{1}{2}})}.
  \end{align}

  From Theorem \ref{thm:1.2} and \eqref{eq:1.14}--\eqref{eq:1.15}, we
  have the following result.
  \begin{corollary}\label{cor:1.3}
  If initial data $(u_0-\bar{u},{v}_0)$ satisfying $(u_0-\bar{u},\nabla\ln {v_0})\in {H}^{2}(\mathbb R^3)
  \times{H}^{1}(\mathbb R^3)$
  and if there exists constant $\varepsilon_0>0$ such that $\|(u_0-\bar{u},\nabla\ln v_0)\|_{H^2\times{H}^{1}}\le\varepsilon_0$, then
  system \eqref{eq:1.1} has a global solution $(u,v)$ satisfying $\|(u-\bar{u},\nabla{\ln v})\|_{L^{\!\infty}_tH^2\times{L^{\!\infty}_tH}^1}\!\lesssim\varepsilon_0$, $$(u-\bar{u},\nabla\ln{v})\in {C}([0,\infty)\,;H^2)\times
   {C}([0,\infty)\,;H^1)$$
 and $\sup_{t>0}\,(1\!+t)^{\frac{1}{2}}\|(\nabla{u},\Delta{\ln}v)\|_{L^2\times L^2}+\sup_{t>0}\,(1\!+t)^{\frac{7}{8}}\|\Lambda^{\frac{7}{4}}u\|_{L^2}
  \lesssim\varepsilon_0$ and moreover
 \begin{align}\label{eq:1.16}
  \|u-\bar{u}\|_{L^\infty}\lesssim
  (1+t)^{-\frac{1}{2}}\text{ as } t\rightarrow \infty\,;\quad
  \|v\|_{L^\infty}\sim{e}^{(\bar{u}-\mu)t}\text{ as } t\rightarrow \infty.
 \end{align}
\end{corollary}

 %------------------------------------------------------------------------------------------------------------

%%-------------------------------------------------------------------------------------------------------------------------------
\medskip

\noindent \textit{Plan of the paper:} In Sect. \!2 we introduce
several preliminaries lemmas, while
 in Sect. \!3 we prove Theorems \ref{thm:1.1} and \ref{thm:1.2} and Corollary \ref{cor:1.3}.
%--------------------------------------------------------------------------------------------------------------------
%                                     definition        1.1
%--------------------------------------------------------------------------------------------------------------------
 \section{Preliminary lemmas}
 In this section, we list several known lemmas and prove some key lemma which will be used in proving the well-posedness of the parabolic-hyperbolic
 chemotaxis. The first lemma given below is concerned with functions whose Fourier transforms are  supported in low, medium and high frequency areas in the
 frequency space. We note that the first two results are the well-known Bernstein's inequalities (cf. \cite{LemarierRieusset:2002} Proposition 3.2 on page 24, or
 \cite{BahouriCheminDanchin:2011} Lemma 2.1 on page 52) and the last one
 is a direct applications of the Sobolev embedding theorem.
 \begin{lemma}\label{lem2-1}
 If $(s,a,b)\in[0,\infty)\times[1,\infty]^2$,
  $a\le{b}$ and $f(x)\in{L}^a$, then for any two positive constants $c_1$ and $c_2$ there exists positive constant
  $c$ such that
 \begin{align}\label{eq:2.1}
  &\supp\; \widehat{f}\! \subset \{\xi\in\mathbb{R}^3;\; |\xi|\le c_2\},\hskip1.42cm\|\Lambda^{{s}}{f}\|_{L^b} \le c\;\!c_2^{s+{n}(\frac{1}{a}-\frac{1}{b})}\,\|f\|_{L^a},\\
  &\supp\; \widehat{f}\! \subset \{\xi\in\mathbb{R}^3;\; c_1\!<\!|\xi|\!<\!c_2\},\quad\quad\frac{c}{\kappa^{c_1^{\,s}}}\|{f}\|_{L^a}\le \|\Lambda^s f\|_{L^a}
             \le {c}\kappa^{c_2^{\,s}}\|f\|_{L^a},\label{eq:2.2}\\
  &\supp\; \widehat{f}\! \subset \{\xi\in\mathbb{R}^3;\; |\xi|\ge c_1\},\hskip1.43cm\|{f}\|_{L^a}  \le \|\langle\Lambda\rangle^{s}f\|_{L^a}=\|f\|_{W^{s,p}}, \label{eq:2.3}
 \end{align}
 where $\kappa=\ln\frac{c_2}{c_1}$ and $W^{s,p}$ is the fractional Sobolev space.
\end{lemma}
\begin{proof}
 The first two results are direct consequences of Proposition 3.2 of \cite{LemarierRieusset:2002} and Lemma
 2.1 of \cite{BahouriCheminDanchin:2011} by using Littlewood-Paley
 decomposition, while the third inequality is also a direct consequence of Sobolev embedding theorem. Hence we finish
 the proof.
\end{proof}

%------------------------------------------------------------------------------------------------------------------------------------
 Applying Lemma \ref{lem2-1} with $2=a\le{b}\le\infty$ and $s\ge0$ to $\eta(D)f$ and $\psi(D)f$, we
 get
 $$\|\eta(D)f\|_{L^b}\lesssim \|f\|_{L^2}\;\text{ and }\; \|\psi(D)f\|_{L^2}\lesssim \|\Lambda^sf\|_{L^2}.$$
 From \eqref{eq:2.1}--\eqref{eq:2.2}, we have
the following lemma concerning the $L^2$ Fourier
 multiplier.
 \begin{lemma}\label{lem:2.2}
 If $\;r\in[1,\infty]$, $v\in{L}^2$, $m(t,\xi)\in\!{L}^{r}_tL^{\!\infty}_\xi$ and $m(t,D)v=\!\mathcal{F}^{-1}m(t,\xi)\widehat{v}(\xi)$, then
 we get
 \begin{align}\label{eq:2.4}
      \|m(t,D)v\|_{L^r_tL^2}\le
      \|m\|_{{L}^{r}_tL^{\!\infty}_\xi}\|v\|_{L^2};
  \end{align}
 if $\;r\!\in\![2,\infty]$, $v\!\in\!{L}^2$\!, $m(t,\xi)|\xi|^{s}\!\in\!{L^{\!\infty}_\xi{L}^r_t}$ and
 $m(t,D)v\!=\!\mathcal{F}^{-1}m(t,\xi)\widehat{v}(\xi)$, then we get
 \begin{align}\label{eq:2.5}
      \|m(t,D)v\|_{L^r_t\dot{H}^{s}}\le \sup_{\xi\in\mathbb{R}^3}\Big(|\xi|^s\|m(\cdot,\xi)\|_{L^r_t}\Big)\|v\|_{L^2}.
  \end{align}
 \end{lemma}
 \begin{proof}
  The proof of \eqref{eq:2.4} follows from the classical Fourier multiplier
  theory and for readers convenience, we give the proof as follows
 \begin{align*}
  \|m(t,D)v\|_{L^r_tL^{2}}&=\|m(t,\cdot)\widehat{v}(\cdot)\|_{L^r_tL^2_\xi} \le \|\|m(t,\cdot)\|_{L^\infty_\xi}\|_{L^r_t}\|\widehat{v}\|_{L^2_\xi}\nonumber\\
  &\le \|m\|_{L^r_tL^\infty_\xi}\|v\|_{L^2}.
 \end{align*}
 In order to prove \eqref{eq:2.5}, we need
  to use Plancherel equality, Minkowski's inequality, H\"{o}lder's inequality and Plancherel equality again, i.e.,
 \begin{align*}
  \|m(t,D)v\|_{L^r_t\dot{H}^{s}}&=\|m(t,\cdot)|\cdot\!|^s\widehat{v}(\cdot)\|_{L^r_tL^2_\xi}
  \lesssim\|m(t,\cdot)|\cdot\!|^s\widehat{v}(\cdot)\|_{L^2_\xi{L^r_t}}\\
  &\le \sup_{\xi\in\mathbb{R}^3}\Big(\|m(\cdot,\xi)\|_{L^r_t}|\xi|^s\Big)\|\widehat{v}\|_{L^2_\xi}.
 \end{align*}
 Hence we finish the proof.
 \end{proof}
  The skill we used in proving Lemma \ref{lem:2.2} will be used repeatedly
 in the following subsections. In this paper, the multipliers satisfying
 the assumptions of Lemma \ref{lem:2.2} are $e^{-ct|\xi|^2}$ and
 $e^{-ct}\frac{1}{1+|\xi|^2}$ as well as $e^{-ct}$.
 The next lemma is devoted to estimate the bilinear term which is known as
  the maximal $L^r_tL^\rho$ regularity result for heat
 kernel (cf. \cite{LemarierRieusset:2002}, Chapter 7).
% \begin{lemma}\label{lem:2-3}
 The operator $A$ defined by
% \begin{align}\label{eq:2.6}
 $$ u(t,x)\mapsto
 Au(t,x)=\int_0^te^{(t-\tau)\Delta}\Delta{u}(\tau,x)d\tau$$
% \end{align}
 is bounded from $L^r_tL^\rho$ to $L^r_tL^\rho$ with
 $1<r,\rho<\infty$.
% \end{lemma}
 In this paper, we also need to establish a
 similar result whose proof is even simpler in Sobolev spaces  and hence we list it as the following lemma.

 \begin{lemma}\label{lem:2.3}
 If $|m(t,\xi)|\le c_1(\frac{e^{-ct}}{1+|\xi|^2}+e^{-ct|\xi|^2})$ and $|\mu(t,\xi)|\le c_1e^{-ct}$ with positive constants $c$ and $c_1$,
 $2 \!\le \!\rho \!\le \!\infty$, $1\le\!{r}\le\!\rho_1\le\!\infty$ and $(u,v)\!\in L^2_tL^2\!\times\!{L}^r_tL^2$, then
 we get
 \begin{align}\label{eq:2.6}
  &\|\int_0^tm(t-\tau,D)u(\tau)d\tau\|_{L^\rho_t\dot{H}^{1+\frac{2}{\rho}}}\lesssim\|u\|_{L^2_tL^2},%+(\sum_{k\ge0}\|\psi(2^{-k}D)u\|_{L^r_tL^2}^2)^{\frac{1}{2}};
 \\ \label{eq:2.7}
  &\|\int_0^t\mu(t-\tau,D)u(\tau)d\tau\|_{L^{\rho_1}_tL^2}\lesssim\|u\|_{L^r_tL^2}.
  \end{align}
 \end{lemma}
 \begin{proof}
 %Proof of \eqref{eq:2.7} are quite straightforward
 By applying Plancherel equality, Lemma \ref{lem:2.2} with $\|m(t,\cdot)\|_{L^\infty_\xi}\le c_1e^{-ct}$
 and integrability of $e^{-ct}$, we see that the proof of \eqref{eq:2.7} is quite straightforward.
 Hence it suffices to prove \eqref{eq:2.6}. Noticing that $1+\frac{2}{\rho}\in[0,2]$, then by making use of definition of Fourier transformation and Fubini theorem, we have
  \begin{align*}
   \|\int_0^tm(t-\tau,D)&u(\tau)d\tau\|_{L^{\rho}_t\dot{H}^{1+\frac{2}{\rho}}}
   =\|\int_0^tm(t-\tau,\xi)|\xi|^{1+\frac{2}{\rho}}\widehat{u}(\tau)d\tau\|_{L^\rho_tL^2_\xi}&{}\\
   &\lesssim\|\int_0^tm(t-\tau,\xi)|\xi|^{1+\frac{2}{\rho}}\widehat{u}(\tau)d\tau\|_{L^2_\xi{L}^\rho_t} &{}\\
   &\lesssim\|\int_0^t(e^{-c(t-\tau)|\xi|^2}|\xi|^{1+\frac{2}{\rho}}+e^{-c(t-\tau)})|\widehat{u}(\tau)|d\tau\|_{L^2_\xi{L}^\rho_t} &{}\\
   &\lesssim\|\|\widehat{u}\|_{L^2_t}\|_{L^2_\xi}\sim\|\widehat{u}\|_{L^2_tL^2_\xi}&{}\\
   &\lesssim\|u\|_{L^2_tL^2},
  \end{align*}
  where in the second, fourth and fifth inequalities we have used
  Minkowski, Young's inequality, Fubini theorem and Plancherel
  equality.
 Hence we finish the proof
 \end{proof}

 %It is assumed that $\rho\in[2,\infty]$ in Lemma \ref{lem:2.3}.
% For any given $u\in L^r_tL^2$ with $1<r \le  \rho<2$, can we
% prove $\|\int_0^tm(t-\tau,D)u(\tau)d\tau\|_{L^\rho_t\dot{H}^{2+\frac{2}{\rho}-\frac{2}{r}}}\lesssim\|u\|_{L^r_tL^2}$?
% We observe that the case $\rho=r$ follows from maximal regularity of heat kernel; but we do not know how to prove the general case $1<r<\rho<2$.
%  Applying Lemma \ref{lem:2.2}, we obtain that the answer is positive provided that $|m(t,\xi)||\xi|^{2+\frac{2}{\rho}-\frac{2}{r}}\in L^{\rho_1}_tL^\infty_\xi$ with
%  $\frac{1}{\rho_1}=1+\frac{1}{\rho}-\frac{1}{r}$.
%  Noticing that
%   $e^{-ct|\xi|^2}|\xi|^{\frac{2}{\rho_1}}\notin{L}^{\!\rho_{1}}_t\!L^\infty_\xi$,
%   if $|\xi|>2^4$ and $t\ge1$, then
%   $$\sup_{|\xi|>2^4}e^{-ct|\xi|^2}|\xi|^{\frac{2}{\rho_1}}\lesssim\sup_{|\xi|>2^4}e^{-\frac{ct}{2}}e^{-\frac{c|\xi|^2}{2}}
%   |\xi|^{\frac{2}{\rho_1}}\lesssim{e}^{-\frac{ct}{2}} \in L^1_t\cap {L}^{\infty}_t;$$
%   else if $\xi\sim 2^{k}$, then
%   $e^{-{c2^{2k}}}2^{\frac{2k}{\rho_1}} \in L^{\rho_1}_t$
%   which is the key point of \cite{CheminLerner:1992314}.

\noindent\;\;\;\;\;The next lemma is about the Picard contraction
 argument (see e.g.\! \cite{Cannone:1995}). We will use this lemma to
 prove the main
 results concerning well-posedness of system \eqref{eq:1.4} with
 $(p_0,q_0)$ being chosen as $(x_{10},x_{20})$
 and initial data space $X_{10}\times{X}_{20}$ being $L^2(\mathbb{R}^3)\times{H}^1(\mathbb{R}^3)$ or $H^2(\mathbb{R}^3)\times H^1(\mathbb{R}^3)$.
\begin{lemma}\label{lem:2.5} \
  Let $({X_{10}}\times {X_{20}},\ \|\cdot\|_{ {X_{10}}} + \|\cdot\|_{ {X_{20}}})$ and
   $({X_1}\times {X_2},\ \|\cdot\|_{ {X_1}} + \|\cdot\|_{ {X_2}})$  be abstract Banach product spaces,
  $L_1:X_{10}\times{X_{20}}\rightarrow X_1$, $L_2:X_{10}\times{X_{20}}\rightarrow{X_2}$, $B_1:
  {X_1}\times {X_2}\rightarrow{X_1}$ and $B_2: {X_1}\times {X_2}\rightarrow{X_2}$
  are two linear and two bilinear operators such that for any $(x_{10},x_{20})\in X_{10}\times{X}_{20}$, $(x_1,x_2)\in{X_1}\times{X_2}$, $c>0$ and $i=1,2$, if
\begin{align*}
 \|L_i(x_{10},x_{20})\|_{X_i}\le{c}(\|x_{10}\|_{ {X_{10}}}+\|x_{20}\|_{{X_{20}}})\ \text{and}\  \|B_i(x_1,x_2)\|_{X_i}\le{c}\|x_1\|_{{X}_1}\|x_2\|_{{X}_2},
\end{align*}
 then for any
 $(x_{10},x_{20})\!\in\! {X_{10}}\!\times\!{X_{20}}$ with $\|(x_{10},{x_{20}})\|_{ {X_{10}}\!\times\!{X_{20}}}\!<\!\frac{1}{4c^2}$, the following system
\begin{equation*}\label{eq4.3-1}
     (x_1, x_2)=(L_1(x_{10}, x_{20}),L_2(x_{10}, x_{20}))
          + (B_1(x_1,x_2),\ B_2(x_1,x_2) )
       \end{equation*}
 has a solution $(x_1,x_2)$ in $ {X_1}\times {X_2}$. In particular, the solution is such that
\begin{align*}
  \|(x_1,{x_2})\|_{ {X_1}\times {X_2}}\le{4c\|(x_{10},{x_{20}})\|_{ {X_{10}}\times {X_{20}}}}
\end{align*}
 and it is the only one such that $\|(x_1,{x_2})\|_{ {X_1}\times {X_2}}<\frac{1}{c}.$
\end{lemma}

 The last lemma is the limiting case of the Sobolev inequality in $BMO$, see \cite{KozonoTaniuchi:2000191}.
\begin{lemma}\label{lem2.6}
 For $n=\!3$ and $s>\!\frac{1}{4}$, there exists a
 constant $C$ depending on $s$ so that
 \begin{align*}
   \|f\|_{L^\infty}\le
   {C}(1+\|f\|_{BMO}(1+\max\{0,\ln{\|f\|_{W^{s,\,12}}}\}))\quad\text{for all}\; f\in W^{s,12}.
 \end{align*}
\end{lemma}

\section{Cauchy problem of parabolic-hyperbolic system \eqref{eq:1.4}}

  In this section, we mainly use Fourier transformation framework to study the well-posedness of \eqref{eq:1.4} with initial data
  in Sobolev the space.
\subsection{Linearization of {\eqref{eq:1.4}} and the corresponding integral equations}

  In this subsection, we first study the linearized system of \eqref{eq:1.4} %with%\footnote{Similar arguments of system
around $(p_0,q_0)$
 \begin{align}\label{eq3,1}
 \frac{d}{dt}
    \left(\!
      \begin{array}{c}
         p \\
         q \\
     \end{array}
   \!\right)
       =
    \left(\!
     \begin{array}{cc}
       -\Lambda^2          &          \Lambda                   \\
       -\Lambda            &              0                     \\
       \end{array}
     \!\right)
       \left(\!
    \begin{array}{c}
       p \\
       q \\
     \end{array}
    \!\right).
    \hskip1.3cm
  \end{align}
 Taking Fourier transform of \eqref{eq3,1} with respect to the space variable yields
 \begin{align*}%\label{eq3,2}
   \frac{d}{dt}
     \left(\!
      \begin{array}{c}
        \widehat{p} \\
        \widehat{q} \\
      \end{array}
     \!\right)
        =
       L(\xi)\left(\!
     \begin{array}{c}
       \widehat{p} \\
       \widehat{q} \\
     \end{array}
     \!\right)
         \quad\text{with } L(\xi)
         =
     \left(\!
      \begin{array}{cc}
      -|\xi|^2       &   |\xi|        \\
       -|\xi|        &   0    \\
     \end{array}
       \!\right).
  \end{align*}

 The characteristic polynomial of $L(\xi)$ is
 $X^2+|\xi|^2X+|\xi|^2$. According to the size of $|\xi|$, we have
 the following three subcases:
 \begin{description}
   \item[]\hskip.8cm$\bullet$\hskip.3cm If $|\xi|>2$, then the characteristic polynomial possesses two distinct {\it real} roots:
     $\lambda_+=\frac{|\xi|^2}{2}(-1+\Xi)$ and $\lambda_-=\frac{|\xi|^2}{2}(-1-\Xi)$ with $\Xi :=\sqrt{1-\frac{4}{|\xi|^2}}$.
    Since $\lambda_1\ne\lambda_2$, the matrix $L(\xi)$ is diagonalizable. After computing the associated eigenspaces, we find
   that
      \begin{align}\label{eq3,2}
         \widehat{p}&=\big(\frac{e^{t\lambda_-}+e^{t\lambda_+}}{2}+\frac{e^{t\lambda_-}-e^{t\lambda_+}}{2\Xi}\big)\widehat{p}_0+
         \frac{e^{t\lambda_+}-e^{t\lambda_-}}{\Xi}\frac{\widehat{q}_0}{|\xi|},\\
         \label{eq3,3}
         \widehat{q}&=\frac{e^{t\lambda_-}-e^{t\lambda_+}}{\Xi}\frac{\widehat{p}_0}{|\xi|}+\big(\frac{e^{t\lambda_-}+e^{t\lambda_+}}{2}+
         \frac{e^{t\lambda_+}-e^{t\lambda_-}}{2\Xi}\big)\widehat{q}_0,
     \end{align}
     where, for simplicity, we denote
     $\frac{e^{t\lambda_+}+e^{t\lambda_-}}{2}$ and
     $\frac{e^{t\lambda_+}-e^{t\lambda_-}}{2\Xi}$ by
     $\Omega_{1,t}(\xi)$ and $\Omega_{2,t}(\xi)$, respectively. Moreover, if there
     is no confusion, we will denote $\Omega_{1,t}(\xi)$ and
     $\Omega_{2,t}(\xi)$ by $\Omega_{1,t}$ and $\Omega_{2,t}$, respectively.

   \item[]\hskip.8cm$\bullet$\hskip.3cm If $|\xi|<2$, then the characteristic polynomial has two distinct {\it complex} roots:
     $\lambda_+=-\frac{|\xi|^2}{2}-i\frac{\Theta|\xi|^2}{2}$ and $\lambda_-=-\frac{|\xi|^2}{2}+i\frac{\Theta|\xi|^2}{2}$ with
     $\Theta :=\sqrt{-1+\frac{4}{|\xi|^2}}$. Noticing that $\lambda_1\ne\lambda_2$, hence the matrix $L(\xi)$ is also diagonalizable. After computing the associated eigenspaces,
           we get
      \begin{align}\label{eq3,4}
          \widehat{p}&=\big(\frac{e^{t\lambda_-}+e^{t\lambda_+}}{2}+\frac{e^{t\lambda_-}-e^{t\lambda_+}}{-2i\Theta}\big)\widehat{p}_0 +
          \frac{e^{t\lambda_+}-e^{t\lambda_-}}{-i\Theta}\frac{\widehat{q}_0}{|\xi|},\\
          \label{eq3,5}
          \widehat{q}&=\frac{e^{t\lambda_-}-e^{t\lambda_+}}{-i\Theta}\frac{\widehat{p}_0}{|\xi|} + \big(\frac{e^{t\lambda_-}+e^{t\lambda_+}}{2} +
          \frac{e^{t\lambda_+}-e^{t\lambda_-}}{-2i\Theta}\big)\widehat{q}_0,
      \end{align}
     where, for simplicity, we denote
     $\frac{e^{t\lambda_+}+e^{t\lambda_-}}{2}$ and
     $\frac{e^{t\lambda_+}-e^{t\lambda_-}}{-2i\Theta}$ by
     $\Omega_{3,t}$ and $\Omega_{4,t}$, respectively.

   \item[]\hskip.8cm$\bullet$\hskip.3cm If $|\xi|=2$, $L(\xi)$ is not diagonalizable. However, this case can be defined via
      $\lim_{|\xi|\rightarrow2^+}$ and $\lim_{|\xi|\rightarrow2^-}$ since the two limits not only exist, but also coincide.
   \end{description}

  \ \ \ Analysis of multipliers in \eqref{eq3,2}--\eqref{eq3,5} is  rewriten into the following five subcases:
 \begin{enumerate}
   \item[]$\bullet$\hskip.3cm
     If $|\xi|>4$, then we obtain that $\frac{\sqrt{3}}{2}\!<\Xi\!<1$,
     $\lambda_+\!=-\frac{2}{1+\Xi}$,
     $\lambda_-=-\frac{(1+\Xi)|\xi|^2}{2}$, $
     \Omega_{2,t}=\frac{e^{t\lambda_+}-e^{t\lambda_-}}{2\Xi}=\frac{e^{-\frac{2t}{1+\Xi}}(1-e^{-t|\xi|^2\Xi})}{2\Xi}$
     and
     $\Omega_{1,t}-\Omega_{2,t}=\frac{e^{t\lambda_-}+e^{t\lambda_+}}{2}+\frac{e^{t\lambda_-}-e^{t\lambda_+}}{2\Xi}=
         \frac{e^{-\frac{t(1+\Xi)|\xi|^2}{2}}}{\frac{2\Xi}{\Xi+1}}-\frac{e^{-\frac{2t}{1+\Xi}}}{\frac{\Xi(\Xi+1)|\xi|^2}{2}}$ which yields that
    \begin{align}
    \label{eq3,6}
    |\Omega_{1,t}-\Omega_{2,t}|\le 2e^{-\frac{t|\xi|^2}{2}}+3e^{-t}\frac{1}{1+|\xi|^{2}}\quad\text{and}\quad |\Omega_{2,t}|\le e^{-t}.
    \end{align}

 \item[]$\bullet$\hskip.3cm If $2<|\xi|\le 4$, then we have $0<\Xi\le\frac{\sqrt{3}}{2}$,
     $\lambda_+\!=-\frac{2}{1+\Xi}$,
     $\lambda_-\!=-\frac{(1+\Xi)|\xi|^2}{2}$ and
      $\Omega_{1,t}=\frac{e^{t\lambda_-}+e^{t\lambda_+}}{2}=
      \frac{e^{-\frac{t(1+\Xi)|\xi|^2}{2}}+{e^{-\frac{2t}{1+\Xi}}}}{2}$.
  Applying ${1-e^{-|x|}} \le {|x|}$ to $\Omega_{2,t}$ and noticing that
  $4\!<\!|\xi|^2\!\le\!
  16$, there holds
    \begin{align}\label{eq3,7}
    |\Omega_{1,t}|\le e^{-t}\quad\text{and}\quad|\Omega_{2,t}|\le 16e^{-\frac{t}{2}}.
    \end{align}

\item []$\bullet$\hskip.3cm
    If $1\le|\xi|<2$, then we obtain that $ 0<\Theta|\xi| \le\sqrt 3$,
     $\lambda_{\pm}\:= -\frac{|\xi|^2}{2}\mp
     i\frac{\Theta|\xi|^2}{2}$, $\Omega_{3,t}=\frac{e^{t\lambda_+}+e^{t\lambda_-}}{2}=
 e^{-\frac{t|\xi|^2}{2}} {\cos{\frac{\Theta|\xi|^2t}{2}}}$  and
    $\frac{\Omega_{4,t}}{|\xi|}=\frac{e^{t\lambda_+}-e^{t\lambda_-}}{-2i\Theta|\xi|}=
       \frac{1}{2}e^{-\frac{t|\xi|^2}{2}}\frac{\sin{\frac{\Theta|\xi|^2t}{2}}}
       {\frac{\Theta|\xi|}{2}}=\frac{1}{2}
       e^{-\frac{t|\xi|^2}{2}}\frac{\sin{\frac{\Theta|\xi|^2t}{2}}}
       {\frac{\Theta|\xi|^2t}{2}}t|\xi|$.
        Applying $|\sin{x}|\le |x|$, $|\cos x|\le1$ to $\Omega_{3,t}$ and $\Omega_{4,t}$,
     we get
     \begin{align}\label{eq3,8}
     \frac{|\Omega_{4,t}|}{|\xi|}\le 4e^{-\frac{t}{4}} ,\ \ |\Omega_{4,t}|\le  8e^{-\frac{t}{4}}\ \ \text{and}\ \ |\Omega_{3,t}|\le
     e^{-\frac{t}{2}}.\end{align}

\item []$\bullet$\hskip.3cm
    If $|\xi|<1$, then we can prove that $\sqrt{3}<\Theta|\xi|<2$,
    $\lambda_{\pm}\:=\!-\frac{|\xi|^2}{2}\mp
    i\frac{\Theta|\xi|^2}{2}$,
   \begin{align}\label{eq3,9}
   \frac{|\Omega_{4,t}|}{|\xi|}\le e^{-\frac{t|\xi|^2}{2}},\ \ |\Omega_{4,t}|\le
     4e^{-\frac{t|\xi|^2}{4}}\ \text{ and } \ |\Omega_{3,t}|\le
     2e^{-\frac{t|\xi|^2}{2}}.
   \end{align}

\item[]$\bullet$\hskip.3cm
   If $|\xi|\!\rightarrow\!2$, then we obtain that
   $\lim_{|\xi|\rightarrow2^+}\Xi=\lim_{|\xi|\rightarrow2^-}\Theta=0$,
   $\lim_{|\xi|\rightarrow2}\lambda_+=\lim_{|\xi|\rightarrow2}\lambda_-=\!-2$ and %$\lim_{|\xi|\rightarrow2}\Xi|\xi|=0$ and
   \begin{align}\label{eq3,10}
     \lim_{|\xi|\rightarrow2^+}\Omega_{1,t}=\lim_{|\xi|\rightarrow2^-}\Omega_{3,t}=e^{-2t},\
     \lim_{|\xi|\rightarrow2^+}\Omega_{2,t}=\lim_{|\xi|\rightarrow2^-}\Omega_{4,t}=
      2te^{-2t}.
   \end{align}
\end{enumerate}

 For simplicity, we define the following two multipliers:
\medskip\vskip.2cm
 \halign{\hskip.2cm $#$&    \qquad  $#$   \hfill \cr
  m_{1}(t,\xi)=\!\left\{ \begin{aligned}
                     & \Omega_{1,t}-\Omega_{2,t}    & \text{if }\hskip.1cm &|\xi|>2, \\
                     & e^{-2t}\!-\!2te^{\!-2t} & \text{if }\hskip.1cm &|\xi|=2,\ \\
                     & \Omega_{3,t}-\Omega_{4,t}    & \text{if }\hskip.1cm &|\xi|<2,
                          \end{aligned} \right.&
  m_{2}(t,\xi)=\!\left\{ \begin{aligned}
                     & {\Omega_{2,t}}               & \text{if }\hskip.1cm  &|\xi|>2, \\
                     &  2te^{-2t}       & \text{if }\hskip.1cm &|\xi|=2,\quad\quad(M) \\
                     & {\Omega_{4,t}}               & \text{if }\hskip.1cm  &|\xi|<2.
                          \end{aligned} \right. \cr
}
\medskip
 Applying \eqref{eq3,6}--\eqref{eq3,10} to $m_1(t,\xi)$ and
$m_2(t,\xi)$, we observe that $m_1(t,\xi)$
 and $m_2(t,\xi)$ are not only radial but
 also continuous with respect to frequency variable $\xi$.
 Moreover, there exist constants $c$ and $c_1$ such that if $|\xi|>2^4$, then
 we get
     \begin{align}\label{eq:m1}
     |m_1(t,\xi)|\le c_1(e^{-ct|\xi|^2} + e^{-ct}\frac{1}{1\!+\!|\xi|^2})\ \text{ and }\ |m_2(t,\xi)|\le
          c_1e^{-ct};
          \end{align}
 if $1< |\xi|< 2^5$, then we get
        \begin{align}\label{eq:m2}
        |m_1(t,\xi)|+|m_2(t,\xi)|\le c_1e^{-ct};
        \end{align}
 else if $|\xi| < 2$, then we get
        \begin{align}\label{eq:m3}
        |m_1(t,\xi)|+|m_2(t,\xi)|+\frac{|m_2(t,\xi)|}{|\xi|}\le
         c_1e^{-ct|\xi|^2}.
         \end{align}

  Next we study system \eqref{eq:1.4} with data $(p_0,q_0)$ and write it into equivalent integral
  equations. Taking Fourier transform of \eqref{eq:1.4} with respect to the space
  variable, applying the well-known Duhamel principle to
  \eqref{eq3,2}--\eqref{eq3,5} and then applying the inverse Fourier transform, we get
      \begin{align}\label{eq:314}
        {p}=&\; m_1(t,D){p}_0+{2m_2(t,D)}\Lambda^{-1}{q}_0\!-\!\!\int_0^t\! m_1(t\!-\!\tau,D)\Lambda{G}(\tau)d\tau,\\
        {q}=&\!-\!{2m_2(t,D)}{\Lambda^{\!\!-1}}{{p}_0}+(m_1(t,D)+2m_2(t,D)){q}_0 \!-\!2\!\!\int_0^t\!\!m_2(t\!-\!\tau,D){G}(\tau)d\tau, \label{eq:315}
     \end{align}
     where $m_1(t,D)$ and $m_2(t,D)$
     are symbols of $m_1(t,\xi)$ and $m_2(t,\xi)$, respectively.
 From \eqref{eq:314} and \eqref{eq:315}, for any $(p_0,q_0)\in L^2\times H^1$, we define a map $\mathfrak{F}$ such that
 \begin{align}\label{eq:316}
  \mathfrak{F}({p,q})&=(\mathfrak{F}_1(p,q),\mathfrak{F}_2(p,q)) =(\text{``r.h.s." of (3.14)},\ \text{``r.h.s." of (3.15)}),
  \end{align}
  where ``r.h.s." stands for ``right hand side".

 The proof of Theorem \ref{thm:1.2} is similar but simpler than that of
 Theorem \ref{thm:1.1}, thus we prove Theorem \ref{thm:1.1} first.
 \subsection{Proof of Theorem \ref{thm:1.1}}
 In this subsection,  we first prove several {\it a priori} estimates including the crucial bilinear estimates.
 We define the corresponding resolution spaces as follows
 \begin{align}\label{eq:317}
   &X\times Y      = \Big\{(p,q)\in\!C([0,\infty);L^2)\!\times\!{C}([0,\infty);H^1)\text{ and} \ \|p\|_X+\|q\|_Y\!<\!\infty\Big\},
  \end{align}
 where $\|p\|_X:=   \|p\|_{L^\infty_tL^2} + \| p\|_{L^2_t\dot{H}^1}
                  + \|p\|_{{L}^{1}_t\dot{H}^{\frac{7}{4}}_\psi}\;\text{ and } \;
        \|q\|_Y:=   \|q\|_{L^\infty_t{H}^{1}}+ \|q\|_{L^2_t\dot{H}^1}$.

 In what follows, we prove several key estimates.
 \begin{proposition}\label{pro:3.1}%
  Let $(p,q)$ be a solution to system \eqref{eq:1.4} with $(p_0,q_0)\in L^2(\mathbb R^3)\times
  H^1(\mathbb R^3)$ and $\mathfrak{F}$ and $\mathfrak{F_1}$  be defined as in \eqref{eq:316}. Then there hold
 \begin{align}
        \label{eq3,18}
  &\|\mathfrak{F}({p,q})\|_{L^{\!\infty}_tL^2\times{L}^{\!\infty}_tH^1}
         \lesssim    \|(p_0,q_0)\|_{L^2\times{H}^1}\! + \|G\|_{L^2_tL^2}
        + \|G\|_{{L}^{1}_t\dot{H}^{1}},\\
        \label{eq3,19}
  &\|\mathfrak{F}({p,q})\|_{L^2_t\dot{H}^1\times{L^2_t\dot{H}^1}}
        \hskip.08cm \lesssim     \|(p_0,q_0)\|_{L^2\times{H}^1} + \|G\|_{L^2_tL^2} + \|G\|_{L^{1}_t\dot{H}^1},\\
        \label{eq:320}
  &\|\mathfrak{F}_1({p,q})\|_{{L}^{1}_t\dot{H}^{\frac{7}{4}}_\psi}
        \lesssim      \|(p_0,q_0)\|_{L^2\times{H}^1} + \|G\|_{{L}^{1}_t\dot{H}^{1}}.
 \end{align}
 \end{proposition}
 \begin{proof}
   In order to prove \eqref{eq3,18}--\eqref{eq:320}, from
   \eqref{eq:315}--\eqref{eq:317} we observe that we have to
   establish several estimates whose proof will be divided into three parts.

   \noindent \text{Part I}. \ {\it Estimate of
   $\|\mathfrak{F}({p,q})\|_{L^\infty_tL^2\times{L^\infty_tH^1}}$}.

   First, we derive the estimate for $\mathfrak{F_1}({p,q})$ defined in \eqref{eq:314} \eqref{eq:316}.

   Noticing that any $L^\infty_\xi$ function $m(\xi)$ is an $H^s$ (or $\dot{H}^s$) Fourier multiplier which means that
   for any $H^s$ (or $\dot{H}^s$) function $f(x)$ (or $g$), there hold
   \begin{align}\label{eq:321}
   \left\{\begin{aligned}&\|m(D){f}\|_{H^s}=\|\langle\cdot\rangle^sm(\cdot)\widehat{f}(\cdot)\|_{L^2_\xi}\lesssim\|m\|_{L^\infty_\xi}\|f\|_{H^s},\\
   &\|m(D){g}\|_{\dot{H}^s}=\||\cdot\!|^sm(\cdot)\widehat{g}(\cdot)\|_{L^2_\xi}\lesssim\|m\|_{L^\infty_\xi}\|g\|_{\dot{H}^s},
   \end{aligned}\right.
   \end{align}
   where $\|\langle\cdot\rangle^{s}\widehat{f}(\cdot)\|_{L^2_\xi}=\|f\|_{H^s}$
   and
   $\||\!\cdot\!|^s\widehat{f}(\cdot)\|_{L^2_\xi}=\|f\|_{\dot{H}^s}$.

   For $m_1(t,D)$ and $m_2(t,D)\Lambda^{-1}$, from \eqref{eq:m1}--\eqref{eq:m3} and a simple
   calculation, we have that $m_{1}(t,\xi),\frac{2m_{2}(t,\xi)}{|\xi|}\in L^\infty_tL^\infty_\xi$. Hence by applying
   \eqref{eq:2.4} with $r=\infty$ and $s=0$ to $m_1(t,D)p_0+2m_2(t,D)\Lambda^{-1}q_0$, we get
  \begin{align}\label{eq322}
        \|m_1(t,D)p_0+2m_{2}(t,D)\Lambda^{-1}q_0\|_{L^\infty_tL^2}
       &\lesssim\|p_0\|_{L^2}+\|q_0\|_{L^2}.
   \end{align}
  As for $m_1(t,D)+2m_2(t,D)$ and $m_2(t,D)\Lambda^{-1}\langle\Lambda\rangle$, from
  \eqref{eq:m1}--\eqref{eq:m3} and \eqref{eq322}, we observe that $m_{1}(t,\xi)+2m_2(t,\xi),\frac{\langle\xi\rangle{m}_{2}(t,\xi)}{|\xi|}\in L^{\!\infty}_tL^{\!\infty}_\xi$.
   Hence  applying
   \eqref{eq:2.4}  with $r=\infty$ and $s=1$ to $m_2(t,D)\Lambda^{-1}p_0+(m_1(t,D)+2m_2(t,D))q_0$, we get
  \begin{align}\label{eq323}
      &\|-2m_2(t,D)\Lambda^{-1}p_0+(m_1(t,D)+2m_2(t,D))q_0\|_{L^\infty_tH^1}\nonumber\\
     &\lesssim \|2m_2(t,D)\Lambda^{-1}\langle\Lambda\rangle p_0\|_{L^\infty_tL^2}
       + \|(m_1(t,D)+2m_2(t,D)){q}_0\|_{L^\infty_tH^1}\nonumber\\
     &\lesssim\|p_0\|_{L^2}+\|q_0\|_{H^1}.
   \end{align}

\noindent  Then we deal with the third term of \eqref{eq:314}.
Applying \eqref{eq:m1}--\eqref{eq:m3} and  \eqref{eq:2.7} with
$r=2$, $\rho=\infty$, $s=1$ and
  $m(t,\xi)=m_1(t,\xi)$ to $G$, we get
  \begin{align}\label{eq324}
     \|\int_0^t\!m_1(t\!-\!\tau,D)\Lambda G(\tau)d\tau\|_{L^\infty_tL^2}
     & =         \int_0^t\!m_1(t\!-\!\tau,D)G(\tau)d\tau\|_{L^\infty_t\dot{H}^1}\nonumber\\
     & \lesssim \|G\|_{L^2_tL^2}.
   \end{align}
  It remains to derive the estimate for $\mathfrak{F_2}({p,q})$ defined in \eqref{eq:315} \eqref{eq:316}. By partition of unit, we
  have $G=G^l+G^m+G^h$. Then from  Lemma \ref{lem2-1}, we get
   \begin{align*}
   & \|\int_0^t\!m_2(t-\tau,D)G(\tau)d\tau \|_{L^\infty_tH^1}
                      = \|\int_0^t\!m_2(t-\tau,D)\langle\Lambda\rangle G(\tau)d\tau\|_{L^\infty_tL^2}\\
   &\le\|\int_0^t(m_2(t-\tau,D)\Lambda^{-1})\Lambda\langle\Lambda\rangle(G^l+G^m)d\tau
     + m_2(t-\tau,D)\langle\Lambda\rangle{G}^hd\tau\|_{L^\infty_tL^2}\\
                     &\le\|\!\int_0^t\!\!m_2(t-\tau,D)\Lambda^{-1}\langle\Lambda\rangle({G}^l\!+\!{G}^m)d\tau\|_{L^\infty_t\dot{H}^1} +
                       \|\!\int_0^t\!\!m_2(t-\tau,D){G}^hd\tau\|_{L^\infty_tH^1}\nonumber\\
                     &:=I_{11}+I_{12}.
   \end{align*}
  Applying \eqref{eq:m1}--\eqref{eq:m3}, \eqref{eq:2.7}  and Bernstein inequalities to $I_{11}$ and $I_{12}$, we get
    \begin{align}\label{eq325}
   I_{11}& = \|\int_0^tm_2(t-\tau,D)\Lambda^{-1}\langle\Lambda\rangle({G}^l+{G}^m)d\tau\|_{L^\infty_t\dot{H}^1}\nonumber\\
         & \le\|\langle\Lambda\rangle({G}^l+{G}^m)\|_{L^2_tL^2}\nonumber\\
         & \lesssim \|\langle\cdot\rangle\eta(\cdot)+\langle\cdot\rangle\varphi(\cdot)\|_{L^\infty_\xi}\|G\|_{L^2_tL^2}\nonumber\\
         & \lesssim\|G\|_{L^2_tL^2}
   \end{align}
 and
    \begin{align}\label{eq326}
   I_{12}
     &  =\|\int_0^tm_2(t-\tau,D)\langle\Lambda\rangle{G}^hd\tau\|_{L^\infty_tL^2}\hskip1.95cm\nonumber\\
     & \le\|\Lambda^{-1}\langle\Lambda\rangle\psi(D){G}\|_{L^{1}_t\dot{H}^1}\nonumber\\
     & \le \|\||\cdot|^{-1}\langle\cdot\rangle\psi(\cdot)\|_{L^\infty_\xi}\|G\|_{\dot{H}^1}\|_{L^{1}_t}\nonumber\\
     & \lesssim\|G\|_{L^1_t\dot{H}^1}
   \end{align}
 where in \eqref{eq326} we have used the fact that $0\le\frac{\langle\xi\rangle\psi(\xi)}{|\xi|}\le 2$.

 \noindent \text{Part II}. {\it Estimate of $\|\mathfrak{F}({p,q})\|_{L^2_t\dot{H}^1\times{L^2_t\dot{H}^1}}$.}

We first derive the estimate for $\mathfrak{F_1}({p,q})$ defined in \eqref{eq:314} \eqref{eq:316}.

For $m_1(t,D)$ and $m_2(t,D)\Lambda^{-1}$, from
\eqref{eq:m1}--\eqref{eq:m3} and a simple
   calculation, we observe that $|\xi|m_{1}(t,\xi)+{2m_{2}(t,\xi)}\in L^\infty_\xi{L}^2_t$. Hence by applying
   \eqref{eq:2.5} with $r=2$ and $s=1$ to $m_1(t,D)p_0+2m_2(t,D)\Lambda^{-1}q_0$, we get
  \begin{align}\label{eq327}
        \|m_1(t,D)p_0+2m_{2}(t,D)\Lambda^{-1}q_0\|_{L^2_t\dot{H}^1}
       &\lesssim\|p_0\|_{L^2}+\|q_0\|_{L^2}.
   \end{align}
  As for $m_1(t,D)+2m_2(t,D)$ and $m_2(t,D)\Lambda^{-1}$, from
  \eqref{eq:m1}--\eqref{eq:m3} and \eqref{eq322}, we observe that $|\xi|m_{1}(t,\xi)\in L^\infty_t{L}^2_{\xi}$
   and $m_{2}(t,\xi)\in L^\infty_{\xi}{L^2_t}$.
   Hence applying \eqref{eq:2.5}
    with $r=2$ and $s=1$ to $m_2(t,D)\Lambda^{-1}p_0+(m_1(t,D)+2m_2(t,D))q_0$, we get
  \begin{align}\label{eq328}
       &\|-2m_2(t,D)\Lambda^{-1}p_0+m_1(t,D)q_0+2m_2(t,D)q_0\|_{L^2_t\dot{H}^1}\nonumber\\
       &\lesssim\|p_0\|_{L^2}+\|q_0\|_{L^2}+\|\Lambda{q}_0\|_{L^2}\nonumber\\
       &\lesssim\|(p_0,q_0)\|_{L^2\times{H}^1}.
   \end{align}
  We first deal with the third term on the r.h.s. of \eqref{eq:314}. Applying \eqref{eq:m1}--\eqref{eq:m3} and  \eqref{eq:2.7}
  with $r=2$, $\rho=2$, $s=2$ and
  $m(t,\xi)=m_1(t,\xi)$ to the term of $G$, we get
  \begin{align}\label{eq329}
      \|\int_0^tm_1(t-\tau,D)\Lambda G(\tau)d\tau\|_{L^2_t\dot{H}^1}
     & =\|\int_0^t\!m_1(t-\tau,D)G(\tau)d\tau\|_{L^2_t\dot{H}^2}\nonumber\\
     & \lesssim\|G\|_{L^2_tL^2}.
   \end{align}
  It remains to derive the estimate for $\mathfrak{F_2}({p,q})$ defined in \eqref{eq:315} \eqref{eq:316}. %estimate \eqref{eq:315}.
  %By partition of unit, we have $G=G^l+G^m+G^h$.
  Using similar ways in proving \eqref{eq325} and \eqref{eq326}, we get
   \begin{align}\label{eq330}
    \|\int_0^t\!m_2(t\!-\!\tau,D)G(\tau)d\tau \|_{L^2_t\dot{H}^1}
     &= \|\int_0^t\!m_2(t\!-\!\tau,D)\Lambda{G}(\tau)d\tau\|_{L^2_tL^2}\nonumber\\
     &\lesssim\|G\|_{L^2_tL^2}+\|G^h\|_{L^{1}_t\dot{H}^1}\nonumber\\
     &\lesssim\|G\|_{L^2_tL^2}+\|G\|_{L^{1}_t\dot{H}^1}.
   \end{align}

 \noindent \text{Part III}. {\it Estimate of $\|\mathfrak{F}_1({p,q})\|_{L^{1}_t\dot{H}^{\frac{7}{4}}_\psi}$.}

 From maximal regularity results, \eqref{eq:314} and \eqref{eq:316}, we observe that
 \begin{align}\label{eq331}
    \|\mathfrak{F}_1(p,q)\|_{L^1_t\dot{H}^{\frac{7}{4}}_\psi}\le
            &\ \|m_1(t,D)p_0\|_{L^1_t\dot{H}^{\frac{7}{4}}_\psi} +2\|m_2(t,D)\Lambda^{-1}q_0\|_{L^1_t\dot{H}^{\frac{7}{4}}_\psi}\nonumber\\
            &+\|\int_0^tm_1(t-\tau)\Lambda{G}(\tau)d\tau\|_{L^1_t\dot{H}^{\frac{7}{4}}_\psi}\nonumber\\
            :=&\ I_{21}+I_{22}+I_{23}.
 \end{align}
 As for $I_{21}$, applying \eqref{eq:m1}--\eqref{eq:m2} and Lemma \ref{lem2-1} to $I_{21}$ with $|\xi|\ge2^4$, we
 claim that
 \begin{align}\label{eq332}
  \hskip1.4cm\, I_{21}&=\|m_{1}(t,D)p_0\|_{L^1_t\dot{H}^{\frac{7}{4}}_\psi}\nonumber\\
         &\lesssim \|e^{-ct|\xi|^2}|\xi|^{\frac{7}{4}}\psi(\xi)\widehat{p}_0\|_{L^1_tL^2_\xi}
  +\|e^{-ct}\psi(\xi)|\xi|^{\frac{7}{4}}\langle\xi\rangle^{-2}\|_{L^1_tL^\infty_\xi}\|\widehat{p}_0\|_{L^2_\xi}\nonumber\\
  &\lesssim \|p_0\|_{L^2}.&
  \end{align}
 In order to show \eqref{eq332}, it suffices to estimate
 $\|e^{-ct|\xi|^2}|\xi|^{\frac{7}{4}}\psi(\xi)\widehat{p}_0\|_{L^1_tL^2_\xi}$
 as follows
 \begin{align*}
 \|e^{-ct|\xi|^2}|\xi|^{\frac{7}{4}}\psi(&\xi)\widehat{p}_0\|_{L^1_tL^2_\xi}
  = \int_0^1\!(\int_{\mathbb{R}^3}e^{-2ct|\xi|^2}|\xi|^{\frac{7}{2}}\psi(\xi)|\widehat{p}_0|^2d\xi)^{\frac{1}{2}}dt\\
  & +  \int_1^\infty(\int_{|\xi|>2^4}e^{-2ct|\xi|^2}|\xi|^{\frac{7}{2}}\psi(\xi)|\widehat{p}_0|^2d\xi)^{\frac{1}{2}}dt \\
  &:= I_{211}+I_{212}.
 \end{align*}
  Applying Plancherel equality to $I_{211}$, we have $\displaystyle{\sup_{\xi\in\mathbb{R}^3}}e^{-2ct|\xi|^2}|\xi|^{\frac{7}{2}}\psi(\xi)\lesssim
  t^{-\frac{7}{4}}$ and
 \begin{align*}
   I_{211} & =   \int_0^1(\int_{\mathbb{R}^3}e^{-2ct|\xi|^2}|\xi|^{\frac{7}{2}}\psi(\xi)|\widehat{p}_0|^2d\xi)^{\frac{1}{2}}dt\quad\quad\;\;\\
           & \lesssim  \int_0^1t^{-\frac{7}{8}}dt\|\widehat{p}_0\|_{L^2_\xi}\\
           &\lesssim \|p_0\|_{L^2}.
 \end{align*}
 For $|\xi|>2^4$ and $t>1$, we get
 $e^{-2ct|\xi|^2}|\xi|^{\frac{7}{2}}\psi(\xi)\le {e}^{-ct}e^{-c|\xi|^2}|\xi|^{\frac{7}{2}}\psi(\xi)\lesssim{e}^{-ct}$
 and
 \begin{align*}
  I_{212} &
  \lesssim\int_1^\infty(\int_{|\xi|>2^4}e^{-ct}e^{-c|\xi|^2}|\xi|^{\frac{7}{2}}\psi(\xi)|\widehat{p}_0|^2d\xi)^{\frac{1}{2}}dt\\
          &\lesssim\int_1^\infty{e}^{-ct}dt\|\widehat{p}_0\|_{L^2_\xi}\\
          &\lesssim\|p_0\|_{L^2}.
 \end{align*}
 Similarly, applying \eqref{eq:m1}--\eqref{eq:m2} and Lemma \ref{lem2-1} to $I_{22}$, we get
 \begin{align}\label{eq333}
  \hskip1.3cm I_{22}& =        \|m_{2}(t,D)\Lambda^{-1}q_0\|_{L^1_t\dot{H}^{\frac{7}{4}}_\psi} \lesssim \|e^{-ct}\psi(\xi)\|_{L^1_tL^\infty_\xi}\|q_0\|_{\dot{H}^{\frac{3}{4}}}\nonumber\\
        & \lesssim \|q_0\|_{H^1}.
 \end{align}
 It remains to estimate $I_{23}$. Noticing that the multipliers below can be estimated as follows:
 $m_1(t-\tau,\xi)|\xi|^{\frac{7}{4}}\lesssim(t-\tau)^{-\frac{7}{8}}$ and
 $m_{1}(t-\tau,\xi)|\xi|^{\frac{11}{4}}\lesssim(t-\tau)^{-\frac{11}{8}}$, respectively.
  Hence we get
 \begin{align}\label{eq334}
  \hskip1.4cm I_{23}& = \|\int_0^tm_1(t-\tau)\Lambda{G}(\tau)d\tau\|_{L^{1}_t\dot{H}^{\frac{7}{4}}_\psi}\nonumber\\
        & = \|\int_0^tm_1(t-\tau)\Lambda^{\frac{7}{4}}\Lambda\psi(D){G}(\tau)d\tau\|_{L^{1}_tL^2}\nonumber\\
        & \lesssim \int_0^\infty\!\!\!\int_0^t\!\min\Big\{(t-\tau)^{-\frac{7}{8}}\|\Lambda\psi(D){G}(\tau)\|_{L^2},\;(t-\tau)^{-\frac{11}{4}}\|\psi(D)G\|_{L^2}\Big\}d\tau{dt}\nonumber\\
        & \lesssim \int_0^\infty\!\!\!\int_0^t\!\min\Big\{(t-\tau)^{-\frac{7}{8}},\;(t-\tau)^{-\frac{11}{4}}\Big\}\|{G}(\tau)\|_{\dot{H}^1}d\tau{dt}\nonumber\\
        & \lesssim \int_0^\infty\!\!\!\int_{\tau}^\infty\!\min\Big\{(t-\tau)^{-\frac{7}{8}},\;(t-\tau)^{-\frac{11}{4}}\Big\}dt\|{G}(\tau)\|_{\dot{H}^1}d\tau\nonumber\\
        & \lesssim \int_0^\infty\!\!\!\int_{0}^\infty\!\min\Big\{t^{-\frac{7}{8}},\;t^{-\frac{11}{4}}\Big\}dt\|{G}(\tau)\|_{\dot{H}^1}d\tau\nonumber\\
        & \lesssim \|G\|_{L^1_t\dot{H}^1}
 \end{align}
 where in the fourth inequality we have applied \eqref{eq:2.3} to $\psi(D)G$ with $s=1$ and $a=2$.

 Combining the above arguments, we finish the proof.
\end{proof}

 Recalling that $G=\Lambda^{-1}\nabla\cdot(p\nabla\Lambda^{-1}q)$ and
 Riesz transforms are bounded in $L^2$, thus we need to estimate
 $\|\nabla\cdot(p\nabla\Lambda^{-1}q)\|_{L^1_tL^2}=\|G\|_{L^1_t\dot{H}^1}$.
 The following key lemma is devoted to estimating
 $\|\nabla{p}\cdot\nabla\Lambda^{-1}{q}\|_{L^1_tL^2}$ and $\|p\Lambda{q}\|_{L^1_tL^2}$, where
 \begin{align}\label{eq335}
 \nabla\cdot(p\nabla\Lambda^{\!-1}{q})=\nabla{p}\cdot\nabla\Lambda^{\!-1}{q}-p\,\Lambda{q}.
 \end{align}
\begin{lemma}\label{lem:3.2}
Let $X\times Y$ be  defined in
 \eqref{eq:317}. If $u\in X$ and $v\in Y$, then we get
 \begin{align}\label{eq336}
  &\|u\nabla{v}\|_{L^1_tL^2}+\|{\nabla}uv\|_{L^1_tL^2} \lesssim \|u\|_{L^2_t\dot{H}^1}\|v\|_{L^2_t\dot{H}^1}
   +\|u\|_{L^1_t\dot{H}^{\frac{7}{4}}_\psi}\|v\|_{L^\infty_t{H}^1}, \\
 \label{eq337}
  &\|uv\|_{L^2_tL^2}\lesssim \|u\|_{L^2_t\dot{H}^1}\|v\|_{L^\infty_tH^1}.
\end{align}
\end{lemma}
\begin{proof}
 At first, we prove \eqref{eq336}. Recall that $u{\nabla}v=(u^l+u^m){\nabla}v+u^h{\nabla}v$. By making use of
 H\"{o}lder's inequality, we have
 \begin{align*}
   \|u{\nabla}v\|_{L^1_tL^2}
  & \le\|(u^l+u^m){\nabla}v\|_{L^1_tL^2}+\|u^h{\nabla}v\|_{L^1_tL^2}\\
  & \lesssim \|u^l+u^m\|_{L^{2}_tL^\infty}\|{\nabla}v\|_{L^2_tL^2} +
           \|u^h\|_{L^{1}_tL^{\infty} }\|{\nabla}v\|_{L^{\infty}_tL^2}\\
  :\!&=I_{31}+I_{32}
  \end{align*}
 where from \eqref{eq:2.1} with $|\xi|<2^5$ and Sobolev embedding theorem,
 there holds
 \begin{align}\label{eq:338}
   I_{31}& = \|u^l+u^m\|_{L^{2}_tL^\infty}\|{\nabla}v\|_{L^2_tL^2}\nonumber\\
         & \lesssim\|u^l+u^m\|_{L^{2}_tL^{6}}\|v\|_{L^2_t\dot{H}^1}\nonumber\\
         & \lesssim\|u^l+u^m\|_{L^{2}_t\dot{H}^{1}}\|v\|_{L^2_t\dot{H}^1}\nonumber\\
         & \lesssim\|u\|_{L^{2}_t\dot{H}^{1}}\|v\|_{L^2_t\dot{H}^1}
 \end{align}
 where in the fourth inequality, we used the fact that
 $\eta(\xi)+\varphi(\xi)$ is an $L^2$-multiplier;
 From \eqref{eq:2.3} with $|\xi|>2^4$ and Sobolev embedding theorem $H^{\frac{7}{4}}\hookrightarrow L^\infty$,
 we get
 \begin{align}\label{eq:339}
   I_{32}& = \|u^h\|_{L^{1}_tL^\infty}\|{\nabla}v\|_{L^\infty_tL^2}\nonumber\\
         & \lesssim \|u^h\|_{L^{1}_t{H}^{\frac{7}{4}}}\|v\|_{L^\infty_tH^1}\nonumber\quad\quad\quad\\
          &\lesssim\|u\|_{L^{1}_t\dot{H}^{\frac{7}{4}}_\psi}\|v\|_{L^\infty_tH^1}.
 \end{align}
 Estimate of ${\|\nabla}uv\|_{L^1_tL^2}$ is rather simple. By
 making use of H\"{o}lder's inequality, we get
 \begin{align}\label{eq:3.40}
   \|{\nabla}uv\|_{L^1_tL^2} & \lesssim \|\nabla{u}\|_{L^{2}_tL^{3}}\|v\|_{L^{2}_tL^{6}}\lesssim
   \|{u}\|_{L^2_t\dot{H}^1}\|v\|_{L^{2}_t\dot{H}^{1}}.
 \end{align}
 This proves \eqref{eq336}.

 It remains to prove \eqref{eq337}. By making use of H\"{o}lder's inequality, we get
\begin{align}\label{eq:341}
 \|uv\|_{L^2_tL^2}& \lesssim \|u\|_{L^2_tL^6}\|v\|_{L^\infty_tL^3}\lesssim\|u\|_{L^2_t\dot{H}^1}\|v\|_{L^\infty_tH^{1}}.
  \end{align}
 Finally, combining \eqref{eq:338}--\eqref{eq:341}, we prove all the desired results.
\end{proof}

 Applying \eqref{eq336} and \eqref{eq337} to $\nabla{p}\cdot\nabla\Lambda^{-1}q-p\,\Lambda{q}$ and
 $\Lambda^{-1}\nabla\cdot(p\nabla\Lambda^{-1}q)$, respectively,
   combining Proposition \ref{pro:3.1}, Lemma \ref{lem:3.2} and \eqref{eq:317}, we have the following {\it
a-priori} estimates.

\begin{corollary}\label{cor:3.3}
 Let $(p,q)$ be a solution to system \eqref{eq:1.4} with $(p_0,q_0)\in {L}^2(\mathbb R^3)\times{H}^1(\mathbb R^3)$ and $\mathfrak{F}$ be defined as in \eqref{eq:316}.
 Then there  holds
 \begin{align*}%\label{eq:343}
   &\|\mathfrak{F}({p,q})\|_{X\times Y}
  \lesssim    \|(p_0,q_0)\|_{L^2\times{H}^1} + \|(p,q)\|_{X\times{Y}}^2.
 \end{align*}
 \end{corollary}
\noindent \textbf{Proof of Theorem \ref{thm:1.1}}. Applying Lemma
\ref{lem:2.5}, Corollary \ref{cor:3.3} and following a standard
fixed point argument, we prove Theorem \ref{thm:1.1} provided that $\|(p_0,q_0)\|_{L^2\times{H}^1}$ is small.

 \subsection{Proof of Theorem \ref{thm:1.2}}
 In this subsection,  we first prove the {\it a priori} estimates including the crucial bilinear
 estimates as follows.
 \begin{proposition}\label{pro:3.4}%
  Let $(p,q)$ be a solution to system \eqref{eq:1.4} with $(p_0,q_0)\in H^(\mathbb R^3)2\times
  H^1(\mathbb R^3)$ and $\mathfrak{F}$ be defined as in \eqref{eq:316}. Then there hold
 \begin{align}
        \label{eq342}
  \|\mathfrak{F}({p,q})\|_{L^{\infty}_tH^2\times{L}^{\infty}_tH^1}  \lesssim    \|(p_0,q_0)\|_{H^2\times{H}^1} + \|p\|_{L^\infty_tH^2}\|q\|_{L^\infty_tH^1}.
 \end{align}
 \end{proposition}
 \begin{proof}
 We first derive the estimate for $\mathfrak{F_1}({p,q})$ defined in \eqref{eq:314} \eqref{eq:316}.

 Applying $m_{1}(t,\xi),\frac{m_{2}(t,\xi)}{|\xi|}\in
 L^\infty_tL^\infty_\xi$, $\frac{\langle\xi\rangle\psi(\xi)}{|\xi|}\in L^\infty_\xi$ and \eqref{eq:2.4} with $r=\infty$ and $s=0$ to $m_1(t,D)p_0+2m_2(t,D)\Lambda^{-1}q_0$, we get
  \begin{align}\label{eq343}
        \|m_1(t,D)p_0+2m_{2}(t,D)\Lambda^{-1}q_0\|_{H^2}
       &\lesssim \|m_1(t,\xi)\langle\xi\rangle^2\widehat{p_0}\|_{L^2_\xi}+\|\frac{m_{2}(t,\xi)\langle\xi\rangle^2}{|\xi|}\widehat{q_0}\|_{L^2_\xi}
       \nonumber\\
       &\lesssim\|p_0\|_{H^2}+\|q_0\|_{H^1}.
   \end{align}
  Similarly, noticing that $m_{1}(t,\xi)+2m_2(t,\xi),\frac{\langle\xi\rangle{m}_{2}(t,\xi)}{|\xi|}\in L^{\!\infty}_tL^{\!\infty}_\xi$,
   applying
   \eqref{eq:2.4}  with $r=\infty$ and $s=1$ to $m_2(t,D)\Lambda^{-1}p_0+(m_1(t,D)+2m_2(t,D))q_0$, we get
  \begin{align}\label{eq344}
      &\|2m_2(t,D)\Lambda^{-1}p_0-(m_1(t,D)+2m_2(t,D))q_0\|_{L^\infty_tH^1}\lesssim\|p_0\|_{L^2}+\|q_0\|_{H^1}.
   \end{align}

\noindent  Now we deal with the third term on the r.h.s. of \eqref{eq:314}.
Applying \eqref{eq:m1}--\eqref{eq:m3} and  \eqref{eq:2.7} with
$r=2$, $\rho=\infty$, $s=1$ and
  $m(t,\xi)=m_1(t,\xi)$ to $G$, we get
  \begin{align}\label{eq345}
   \|\int_0^t\!m_1(t\!-\!\tau,D)G(\tau)d\tau\|_{L^\infty_t\dot{H}^1}&\lesssim \|G\|_{L^\infty_t\dot{H}^{-1}}\lesssim \|G\|_{L^\infty_tL^{\frac{3}{2}}}\nonumber\\
   &\lesssim\|p\|_{L^\infty_tH^2}\|q\|_{L^\infty_tH^1}
   \end{align}
  and
  \begin{align}\label{eq346}
   \|\int_0^t\!&m_1(t\!-\!\tau,D)\Lambda{G}(\tau)d\tau\|_{L^\infty_t\dot{H}^2}= \int_0^t\!m_1(t\!-\!\tau,D)G(\tau)d\tau\|_{L^\infty_t\dot{H}^3}\nonumber\\
   &\lesssim  \|G\|_{L^\infty_t\dot{H}^{1}}\lesssim\|\nabla{p}\|_{L^\infty_tL^6}\|q\|_{L^\infty_tL^3}+\|p\|_{L^\infty_tL^\infty}\|\nabla{q}\|_{L^\infty_tL^2}\nonumber
   \\&
   \lesssim\|p\|_{L^\infty_tH^2}\|q\|_{L^\infty_tH^1}.
   \end{align}
  It remains to derive the estimate for $\mathfrak{F_2}({p,q})$ defined in \eqref{eq:315} \eqref{eq:316}.
  %By partition of unit, we have $G=G^l+G^m+G^h$.
  Using similar ways in proving \eqref{eq325} and \eqref{eq326},
  we get
  \begin{align}\label{eq347}
    \|\int_0^tm_2(t-\tau,D)\Lambda^{-1}\langle\Lambda\rangle({G}^l+{G}^m)d\tau\|_{L^\infty_t\dot{H}^1}\lesssim\|G\|_{L^\infty_t\dot{H}^{-1}}
    \lesssim\|G\|_{L^\infty_tL^{\frac{3}{2}}}
   \end{align}
 and
    \begin{align}\label{eq348}
  \|\int_0^tm_2(t-\tau,D)\langle\Lambda\rangle{G}^hd\tau\|_{L^\infty_tL^2} \lesssim
  \|G\|_{L^{\infty}_tH^1}
   \end{align}
 where we have used the damping property of $G^h$, i.e., $\psi(\xi)m_2(t,\xi)\lesssim
 e^{-ct}$.

 Combining the above arguments, we finish the proof.
\end{proof}

 The following proposition is used to prove decay estimates of
 solutions to \eqref{eq:1.1}.
 \begin{proposition}\label{pro:3.5}%
  Let $(p,q)$ be a solution to system \eqref{eq:1.4} with $(p_0,q_0)\in H^2(\mathbb R^3)\times
  H^1(\mathbb R^3)$ and $\mathfrak{F}$ be defined as in \eqref{eq:316}. Then there hold
 \begin{align*}
  &(1+t)^{\frac{1}{2}}\|\nabla\mathfrak{F}({p,q})\|_{L^2}+(1+t)^{\frac{7}{8}}\|\Lambda^{\frac{7}{4}}\mathfrak{F}_1(p,q)\|_{L^2} \\
  & \lesssim    \|(p_0,q_0)\|_{H^2\times{H}^1} + \sup_{t>0}\Big((1+t)^{\frac{1}{2}}\|(\nabla p,\nabla q)\|_{L^2\times L^2}\Big)^2+\sup_{t>0}\Big((1+t)^{\frac{7}{8}}\|\Lambda^{\frac{7}{4}}p\|_{L^2}\Big)^2.
 \end{align*}
 \end{proposition}
 \begin{proof}
 Noticing that $m_{1}(t,\xi)|\xi|+{m_{2}(t,\xi)}\lesssim
 e^{-ct|\xi|^2}|\xi|+e^{-ct}$, we have
  \begin{align*}
       \|m_1(t,D)\Lambda{p}_0+2m_{2}(t,D)q_0\|_{L^2}
       &\lesssim \|m_1(t,\xi)|\xi|\widehat{p_0}\|_{L^2_\xi}+\|m_{2}(t,\xi)\widehat{q_0}\|_{L^2_\xi}\\
       &\lesssim(1+t)^{-\frac{1}{2}}(\|p_0\|_{H^1}+\|q_0\|_{H^1})
   \end{align*}
 and
   %\begin{align*}
    $   \|m_1(t,D)\Lambda^{\frac{7}{4}}{p}_0+2m_{2}(t,D)\Lambda^{\frac{3}{4}}q_0\|_{L^2}
       %&\lesssim \|m_1(t,\xi)|\xi|^{\frac{7}{4}}\widehat{p_0}\|_{L^2_\xi}+\|m_{2}(t,\xi)|\xi|^{\frac{3}{4}}\widehat{q_0}\|_{L^2_\xi}\\
       \lesssim(1+t)^{-\frac{7}{8}}(\|p_0\|_{H^2}+\|q_0\|_{H^1}).
  $% \end{align*}
   \;Similarly,
  \begin{align*}
    \|2m_2(t,D)p_0-(m_1(t,D)\Lambda+2m_2(t,D)\Lambda)q_0\|_{L^2}\lesssim(1+t)^{-\frac{1}{2}}(\|p_0\|_{H^1}+\|q_0\|_{H^1}).
   \end{align*}

 \noindent  As for the third term on the r.h.s. of \eqref{eq:314}, by using
 $m_{1}(t,\xi)|\xi|\lesssim e^{-ct|\xi|^2}|\xi|+e^{-ct}$,
 chain rule, Plancherel equality and Sobolev embedding, we get
  \begin{align*}
   &\|\int_0^t\!m_1(t\!-\!\tau,D)\Delta{G}(\tau)d\tau\|_{L^2}\\
   &\lesssim
   \int_0^t(t-\tau)^{-\frac{1}{2}}(1+\tau)^{-\frac{5}{4}}d\tau\sup_{\tau>0}(1+\tau)^{\frac{5}{4}}(\|{\nabla}p\|_{L^3}\|q\|_{L^6}+\|\nabla{q}\|_{L^2}\|p\|_{L^\infty})
   \\&\lesssim
   (1+t)^{-\frac{1}{2}}
   \sup_{\tau>0}\Big((1+\tau)^{\frac{3}{4}}\|\Lambda^{\frac{3}{2}}p\|_{L^2}+(1+\tau)^{\frac{7}{4}}\|\Lambda^{\frac{7}{4}}p\|_{L^2}\Big)(1+\tau)^{\frac{1}{2}}\|\nabla
   q\|_{L^2}.
   \end{align*}
 Similarly, we have
  \begin{align*}
   &\|\int_0^t\!m_1(t\!-\!\tau,D)\Lambda^{\frac{7}{4}}\nabla{G}(\tau)d\tau\|_{L^2}\\
   &\lesssim
   \int_0^t(t-\tau)^{-\frac{7}{8}}(1+\tau)^{-\frac{5}{4}}d\tau\sup_{\tau>0}(1+\tau)^{\frac{5}{4}}(\|{\nabla}p\|_{L^3}\|q\|_{L^6}+\|\nabla{q}\|_{L^2}\|p\|_{L^\infty})
   \\&\lesssim
   (1+t)^{-\frac{7}{8}}
   \sup_{\tau>0}\Big((1+\tau)^{\frac{3}{4}}\|\Lambda^{\frac{3}{2}}p\|_{L^2}+(1+\tau)^{\frac{7}{4}}\|\Lambda^{\frac{7}{4}}p\|_{L^2}\Big)(1+\tau)^{\frac{1}{2}}\|\nabla
   q\|_{L^2},
   \end{align*}
and
   \begin{align*}
    &\|\int_0^tm_2(t-\tau,D)\nabla{G}d\tau\|_{L^2}\\&\lesssim
   \int_0^t(t-\tau)^{-\frac{1}{2}}(1+\tau)^{-\frac{5}{4}}d\tau\sup_{\tau>0}(1+\tau)^{\frac{5}{4}}(\|{\nabla}p\|_{L^3}\|q\|_{L^6}+\|\nabla{q}\|_{L^2}\|p\|_{L^\infty})
   \\&\lesssim
   (1+t)^{-\frac{1}{2}}
   \sup_{\tau>0}\Big((1+\tau)^{\frac{3}{4}}\|\Lambda^{\frac{3}{2}}p\|_{L^2}+(1+\tau)^{\frac{7}{4}}\|\Lambda^{\frac{7}{4}}p\|_{L^2}\Big)(1+\tau)^{\frac{1}{2}}\|\nabla
   q\|_{L^2}.
   \end{align*}
 Combining the above arguments and  $\|\Lambda^{\frac{3}{2}}p\|_{L^2}\lesssim\|\Lambda^{\frac{7}{4}}p\|_{L^2}^\frac{2}{3}\|\nabla
 p\|_{L^2}^{\frac{1}{3}}$,
 we finish the proof.
\end{proof}

\noindent \textbf{Proof of Theorem \ref{thm:1.2}}. Applying Lemma
\ref{lem:2.5}, Propositions \ref{pro:3.4} and \ref{pro:3.5},
following standard fixed point argument, we prove Theorem
\ref{thm:1.2} provides that $\|(p_0,q_0)\|_{L^2\times{H}^1}$ is
small.

 \subsection{Proof of Corollary \ref{cor:1.3}}
 \noindent \textbf{Proof of Corollary \ref{cor:1.3}}. Applying Lemma
 \ref{lem:2.5} and Corollary \ref{cor:3.3} to system \eqref{eq:1.1},
 we prove the existence results of Corollary \ref{cor:1.3}. As
 for the decay property of $v$, we use \eqref{eq:1.14}--\eqref{eq:1.15}. We
 omit the details.

\medskip\medskip
\noindent{\textbf{Acknowledgmens}:}\,
 Chao Deng is supported by PAPD of Jiangsu Higher Education
 Institutions, by JSNU under Grant No. 9212112101, and by the NSFC under Grant No.\! 11171357 \&\! 11271166; he would like to express his gratitude to Professor
 Congming Li's invitation to the Colorado University at Boulder where part of this
 work was done.
 Tong Li would like to thank Congming Li for his friendship.

\bibliographystyle{amsplain}
\providecommand{\bysame}{\leavevmode\hbox
to3em{\hrulefill}\thinspace}
\providecommand{\MR}{\relax\ifhmode\unskip\space\fi MR }
\providecommand{\MRhref}[2]{%
  \href{http://www.ams.org/mathscinet-getitem?mr=#1}{#2}} \providecommand{\href}[2]{#2}

\medskip\medskip\medskip\medskip\medskip\medskip\medskip
\halign{\hskip 2.4 in  {\small\text{#}}  &    \;\;\;  {\small #}
\hfill \cr
 & Chao Deng   \cr
 & Math Department, Jiangsu Normal University,  \cr
 &Xuzhou, Jiangsu 221116, China  \cr
 &Email: deng315@yahoo.com.cn  \cr\cr
 & Tong Li  \cr
 & Math Department, University of
 Iowa,  \ \cr
 &  Iowa city, IA
52242   \cr
 &  Email: tong-li@uiowa.edu \cr
 }

%\bibliography{references--chao}
\end{document}